\documentclass[12pt,notitlepage]{amsart}
\usepackage{latexsym,amsfonts,amssymb,amsmath,amsthm}
\usepackage{graphicx}
\usepackage{color}
\usepackage[inner=1.0in,outer=1.0in,bottom=1.0in, top=1.0in]{geometry}

\newcommand{\mz}{\ensuremath{\mathbb Z}}

\newcommand{\mh}{\ensuremath{\mathbb H}}

\newcommand{\mc}{\ensuremath{\mathbb C}}

\newcommand{\mymod}{\ensuremath{\negthickspace \negmedspace \pmod}}

\newcommand{\intR}{\int_{-\infty}^{\infty}}

\DeclareMathOperator{\erfc}{erfc}
\DeclareMathOperator{\stabp}{sp}

\theoremstyle{plain}		
	\newtheorem{mytheo}{Theorem} [section]
	
	\newtheorem{myprop}[mytheo]{Proposition}
	\newtheorem{mycoro}[mytheo]{Corollary}
     \newtheorem{mylemma}[mytheo]{Lemma}
	\newtheorem{mydefi}[mytheo]{Definition}
	\newtheorem{myconj}[mytheo]{Conjecture}

\theoremstyle{remark}

\numberwithin{equation}{section}
\numberwithin{figure}{section}

\begin{document}
\title{Zeros of certain combinations of Eisenstein series}
\author{Sarah Reitzes}
\address{Tufts University}
\email{sarah.reitzes@tufts.edu}
\author{Polina Vulakh}
\address{Bard College}
\email{pv2315@bard.edu}
\author{Matthew P. Young} 
\address{Department of Mathematics \\
	  Texas A\&M University \\
	  College Station \\
	  TX 77843-3368 \\
		U.S.A.}
\email{myoung@math.tamu.edu}

\thanks{This work was initiated in summer 2015 during an REU conducted at Texas A\&M University.  The authors thank the NSF and the Department of Mathematics at Texas A\&M for supporting the REU.  In addition, this material is based upon work of M.Y. supported by the National Science Foundation under agreement No. DMS-1401008.  Any opinions, findings and conclusions or recommendations expressed in this material are those of the authors and do not necessarily reflect the views of the National Science Foundation.}

\begin{abstract}
We prove that if $k$ and $\ell$ are sufficiently large, then all the zeros of the weight $k+\ell$ cusp form $E_k(z) E_{\ell}(z) - E_{k+\ell}(z)$ in the standard fundamental domain lie on the boundary.  We moreover find formulas for the number of zeros on the bottom arc with $|z|=1$, and those on the sides with $x = \pm 1/2$.  One important ingredient of the proof is an approximation of the Eisenstein series in terms of the Jacobi theta function.
\end{abstract}

\maketitle
\section{Introduction}
\subsection{Statement of results}
Let $k\geq 4$ be an even integer and let $E_k(z)$ denote the usual holomorphic weight $k$ Eisenstein series defined by
\begin{equation}
\label{eq:EkczdDefinition}
 E_k(z) = \sum_{\gamma \in \Gamma_{\infty} \backslash \Gamma} j(\gamma, z)^{-k} = \frac12 \sum_{(c,d) =1} \frac{1}{(cz + d)^k},
\end{equation}
where $\Gamma = PSL_2(\mz)$, $\Gamma_{\infty} \subset \Gamma$ is the stabilizer of $\infty$, and $j(\gamma, z) = cz + d$ for $\gamma = (\begin{smallmatrix} a & b \\ c & d \end{smallmatrix})$.  

Let $\mathcal{F}$ denote the closure of the standard fundamental domain for $\Gamma \backslash \mathbb{H}$, namely $\mathcal{F} = \{z \in \mh : |z| \geq 1, \text{and } |x| \leq 1/2 \}$.  Rankin and Swinnerton-Dyer \cite{RS} showed all the zeros of $E_k$ in $\mathcal{F}$ lie on the bottom arc $|z|=1$.

In this paper we study the zeros of the weight $k+\ell$ cusp forms defined by
\begin{equation}
 \Delta_{k,\ell} = E_k E_{\ell} - E_{k+\ell}.
\end{equation}
In a few special cases, $\Delta_{k, \ell}$ vanishes identically, namely when (and only when) $k+\ell \in \{8, 10,  14 \}$.  For the rest of this paper we take the convention that $k+\ell \geq 16$ to avoid these trivial cusp forms.
By symmetry we shall assume $k \geq \ell$.

The functions $\Delta_{k, \ell}$ are natural to study for a few reasons.  The most basic motivation comes from the fact that $\Delta_{k, \ell}$ are (arguably) the easiest cusp forms to construct explicitly.  In constrast, the Hecke eigenforms are rather more difficult to construct for large weights.  A more advanced motivation comes from the quantum unique ergodicity (QUE) conjecture of Rudnick and Sarnak \cite{RudnickSarnak}, which is an equidistribution statement for Hecke-Maass cusp forms of large Laplace eigenvalue and for holomorphic Hecke cusp forms of large weight.  Holowinsky and Soundararajan \cite{HolowinskySound} have proven QUE for holomorphic forms.  Meanwhile, QUE for Maass forms was proved by Lindenstrauss \cite{Lindenstrauss} and Soundararajan \cite{SoundQUE}.  In fact, in an even earlier paper, Luo and Sarnak \cite{LuoSarnakQUE} proved QUE for Eisenstein series in the spectral aspect.  One may wonder if there is an analog of QUE for large weight Eisenstein series.  Of course, $E_k^2$ is not a cusp form, so it is natural to project this form onto the cuspidal subspace, which is precisely $\Delta_{k,k}$.  It is an obvious generalization to examine $\Delta_{k, \ell}$ for general $k,\ell$.

Many authors have studied the zeros of modular forms, including \cite{RS}
\cite{AKN} \cite{Kohnen} \cite{Gun} \cite{DukeJenkins}, \cite{Nozaki} \cite{GhoshSarnak}. 

In general we are interested in how these cusp forms $\Delta_{k, \ell}$ behave.  On the one hand, one may expect that $\Delta_{k, \ell}$ has similar properties as a random linear combination of Hecke cusp forms of weight $k+\ell$.  If so, then one would expect that its zeros would equidistribute in the funamental domain as the zeros of Hecke cusp forms do, due to work of Rudnick \cite{Rudnick} and Holowinsky and Soundararajan 
\cite{HolowinskySound}.
In constrast to this naive guess, we have
\begin{myconj} \label{conj:mainconj}
All the zeros of $\Delta_{k, \ell}$ lying in the standard fundamental domain are on the boundary, $|z|=1$ or $x=\pm 1/2$. 
\end{myconj}
This conjecture was born out of numerical evidence for $k,\ell \leq 100$.  Our main result in this paper proves the conjecture for $\ell$ sufficiently large.
\begin{mytheo}
\label{thm:mainthmAllZerosOnBoundary}
There exists an effective constant $L$ so that if $k \geq \ell \geq L$, then all the zeros of $\Delta_{k,\ell}$ in $\mathcal{F}$ lie on its boundary.
\end{mytheo}

The zeros naturally split up into those on the arc, with $|z|=1$, and those on the sides, with $x= \pm 1/2$ (there may also be ``trivial'' zeros at $\rho$ and $i$ which are guaranteed by the valence formula depending on the value of $k+\ell$ modulo $12$, which we do not count in this dicussion).  
One of the curious features of the functions $\Delta_{k, \ell}$ is that the relative proportion of zeros on the arc and on the sides depends strongly on the relative sizes of $k$ and $\ell$.  In particular, in one extreme direction if $k=\ell$ then it appears all the zeros are on the sides.  In the opposite extreme,  if $\ell=4$ and $k \geq 12$, then all the zeros are apparently on the arc.  

Given $k$ and $\ell$, define $A_{k,\ell}$ to be the number of zeros of $\Delta_{k,\ell}$ on the arc $\{ e^{i \theta} : \frac{\pi}{3} < \theta < \frac{\pi}{2} \}$, and $B_{k,\ell}$ to be the number of zeros on the side $x=1/2$, $y > \frac{\sqrt{3}}{2}$.  
The number of zeros of a weight $k+\ell$ modular form is dictated by the valence formula:
\begin{equation}
\label{eq:valenceFormula}
\tfrac{k+\ell}{12}=\tfrac{1}{2}v_{i}(f)+\tfrac{1}{3}v_{\rho}(f)+\underset{\underset{z \in \mathbb{H}}{z\neq i, \rho}}{\sum} v_z(f),
\end{equation}
where $v_{i}(f)$ and $v_{\rho}(f)$ are the orders of vanishing of $f$ at $i$ and $\rho = e^{\pi i/3}$, respectively.  

Thus, Conjecture \ref{conj:mainconj} is equivalent to the formula
\begin{equation}
\label{eq:valenceWithAklBkl}
A_{k,\ell} + B_{k,\ell} = \tfrac{k+ \ell}{12}- \tfrac{1}{2}v_{i}(\Delta_{k,\ell})-\tfrac{1}{3}v_{\rho}(\Delta_{k,\ell}) - 1,
\end{equation}
where the term $-1$ accounts for the zero at $\infty$.  Let us agree to call the zeros at $i$, $\rho$, and $\infty$ guaranteed to exist by the valence formula the trivial zeros, and the rest of the zeros the nontrivial zeros.  

\begin{myconj}
\label{conj:allzerosonarc}
For $k \geq 14$, all of the zeros of $\Delta_{k,4}$, $\Delta_{k,6}$, and $\Delta_{k,8}$ in the fundamental domain lie on the arc $| z | = 1$.
\end{myconj}

\begin{myprop}
\label{prop:nozerosonarc}
For $\ell$ large enough, each of $\Delta_{\ell,\ell}$, $\Delta_{\ell+2,\ell}$, $\Delta_{\ell+4,\ell}$, $\Delta_{\ell+6,\ell}$, and $\Delta_{\ell+10,\ell}$ have its nontrivial zeros on the sides $x= \pm 1/2$.  For $\ell$ large, $\Delta_{\ell+8,\ell}$ has one nontrivial zero on the arc $|z|=1$, and all the rest on the sides $x = \pm 1/2$.
\end{myprop}

These special cases exemplify a more general pattern that when $ k - \ell$ is large, $\Delta_{k,\ell}$ tends to have more zeros on the arc, while when $k - \ell$ is small, $\Delta_{k,\ell}$ tends to have more zeros on the sides.  

In this paper, we shall produce tight bounds on the sizes of $A_{k,\ell}$ and $B_{k,\ell}$, culminating in exact formulas for these quantities.  Since the formulas are somewhat complicated and break into cases, we shall proceed in stages, giving more accurate estimates as we go.

As a starting point, we state the following
\begin{myprop}
\label{prop:CrudeEquality}
For $k \geq \ell \geq 14$, we have $A_{k,\ell} = \frac{k-\ell}{12} + O(1)$, and $B_{k,\ell} = \frac{2\ell}{12} + O(1)$, where the implied constants are absolute (that is, uniform in $k,\ell$). 
\end{myprop}
Since the total number of zeros of a weight $k+\ell$ modular form is $\frac{k+\ell}{12} + O(1)$, Proposition \ref{prop:CrudeEquality} shows that all but a bounded number of zeros of $\Delta_{k,\ell}$ are on the boundary.  It is much more difficult to show all the zeros of $\Delta_{k,\ell}$ are on the boundary. 
Our method of proof shows the ostensibly weaker lower bounds $A_{k,\ell} \geq \frac{k-\ell}{12} + O(1)$, and $B_{k,\ell} \geq \frac{2 \ell}{12} + O(1)$, but once these lower bounds are established, the corresponding upper bounds follow by \eqref{eq:valenceWithAklBkl}.  

We also have the following explicit lower bounds.
\begin{myprop}
\label{prop:Akllowerbound}
Suppose $k \geq \ell \geq 14$, and $k-\ell = 12n + j$ with $n \geq 1$ and $0 \leq j < 12$.  Then $A_{k,\ell} \geq N_{k, \ell}$, 
 where
 \begin{equation}
  N_{k,\ell} = \begin{cases} n-1, \qquad &j = 0, 2, 6 \\
               n, \qquad &j = 4, 8, 10.
              \end{cases}
 \end{equation}
\end{myprop}

\begin{myprop}
\label{prop:BklLowerBoundAsymptotic}
 Suppose $k \geq \ell \geq L$ where $L$ is large.  Then $B_{k,\ell} \geq T_{k,\ell}$, where 
  \begin{equation}
 T_{k,\ell} = 
  \begin{cases} 
   \lfloor \ell/6 \rfloor - 1, \qquad l\equiv 2, 4 \pmod{6} \\
   \lfloor \ell/6 \rfloor - 2, \qquad l \equiv 0 \pmod{6}.
   \end{cases}
  \end{equation}
\end{myprop}
Propositions \ref{prop:Akllowerbound} and \ref{prop:BklLowerBoundAsymptotic}, together with \eqref{eq:valenceFormula}, immediately imply Proposition \ref{prop:CrudeEquality}.

In a handful of cases, the lower bounds in Propositions \ref{prop:Akllowerbound} and \ref{prop:BklLowerBoundAsymptotic} may be elevated to equalities.  If $j\in \{4, 10\}$ and $\ell \equiv 2 \pmod{6}$, then Propositions \ref{prop:Akllowerbound} and \ref{prop:BklLowerBoundAsymptotic} give $A_{k,\ell} + B_{k,\ell} \geq n + \lfloor \ell/6 \rfloor - 1$.  On the other hand, the valence formula \eqref{eq:valenceFormula} gives
\begin{equation}
 A_{k,\ell} + B_{k,\ell} \leq \tfrac{k-\ell}{12} + \tfrac{2\ell}{12} - \tfrac12 \nu_i(\Delta_{k,\ell}) - \tfrac13 \nu_{\rho}(\Delta_{k,\ell}) - 1.
\end{equation}
If $j=4$ and $\ell \equiv 2 \pmod{6}$, this simplifies as
\begin{equation}
 A_{k,\ell} + B_{k,\ell} \leq n + \lfloor \ell/6 \rfloor + \tfrac{2}{3}  - \tfrac12 \nu_i(\Delta_{k,\ell}) - \tfrac13 \nu_{\rho}(\Delta_{k,\ell}) - 1.
\end{equation}
Furthermore, the valence formula \eqref{eq:valenceFormula} implies $\nu_i(\Delta_{k,\ell}) \equiv 0 \pmod{2}$ and $\nu_{\rho}(\Delta_{k,\ell}) \equiv 2 \pmod{3}$, and thus $A_{k,\ell} + B_{k,\ell} \leq n + \lfloor \ell/6 \rfloor - 1$.  The only way this upper bound is consistent with the lower bounds is if $A_{k,\ell} = n$ and $B_{k,\ell} = \lfloor \ell/6 \rfloor -1$.  A similar argument works for $j=10$.  

However, in many cases, Propositions \ref{prop:Akllowerbound} and \ref{prop:BklLowerBoundAsymptotic} are off by one or two zeros, as will be seen with the aid of Tables \ref{table:aklChart} and \ref{table:bklChart}.
\begin{table}[]
\centering
\caption{Values of $A_{k,\ell}$. 
Entries are underlined to emphasize the behavior 
when $k-\ell \equiv 0 \pmod{12}$ as $k$ is increased and $\ell$ is held constant.}
\label{table:aklChart}
\begin{tabular}{l|l|l|l|l|l|l|l|l|l|l|l|l|l|l|l}
$\ell \thinspace \diagdown \thinspace k$ & 56 & 58 & 60 & 62 & 64 & 66 & 68 & 70 & 72 & 74 & 76 & 78 & 80 & 82 & 84 \\
\hline
20  & \underline{3}  & 3  & 3  & 3  & 4  & 3  & \underline{4}  & 4  & 4  & 4  & 5  & 4  & \underline{5}  & 5  & 5  \\
\hline
22  & 2  & \underline{3}  & 2  & 3  & 3  & 3  & 3  & \underline{4}  & 3  & 4  & 4  & 4  & 4  & \underline{4}  & 4  \\
\hline
24  & 2  & 2  & \underline{3}  & 2  & 3  & 3  & 3  & 3  & \underline{4}  & 3  & 4  & 4  & 4  & 4  & \underline{5} 
\end{tabular}
\end{table}

Although it is not easy to directly compare Table \ref{table:aklChart} to Proposition \ref{prop:Akllowerbound}, one \emph{can} certainly read off the fact that Propositions \ref{prop:Akllowerbound} and \ref{prop:BklLowerBoundAsymptotic} are not equalities in all cases.  The reason is that if $\ell$ is held fixed and $k$ is increased by $12$, then $\Delta_{k+12, \ell}$ has one more zero than $\Delta_{k,\ell}$, and the naive expectation is that $A_{k+12,\ell} = A_{k,\ell} + 1$, and $B_{k+12,\ell} = B_{k,\ell}$.  However, the tables give that $A_{82,22} = A_{70,22}$, and $B_{82,22} = B_{70,22} + 1$, for instance.

\begin{table}[]
\centering
\caption{Values of $B_{k,\ell}$. 
Entries areunderlined to emphasize the behavior 
when $k-\ell \equiv 0 \pmod{12}$ as $k$ is increased and $\ell$ is held constant.}
\label{table:bklChart}
\begin{tabular}{l|l|l|l|l|l|l|l|l|l|l|l|l|l|l|l}
$\ell \thinspace \diagdown \thinspace k$ & 56 & 58 & 60 & 62 & 64 & 66 & 68 & 70 & 72 & 74 & 76 & 78 & 80 & 82 & 84 \\
\hline
20  & \underline{2}  & 2  & 2  & 2  & 2  & 2  & \underline{2}  & 2  & 2  & 2  & 2  & 2  & \underline{2}  & 2  & 2  \\
\hline
22  & 3  & \underline{2}  & 3  & 3  & 2  & 3  & 3  & \underline{2}  & 3  & 3  & 2  & 3  & 3  & \underline{3}  & 3  \\
\hline
24  & 3  & 3  & \underline{3}  & 3  & 3  & 3  & 3  & 3  & \underline{3}  & 3  & 3  & 3  & 3  & 3  & \underline{3} 
\end{tabular}
\end{table}

We gather a refined bound with the following.
\begin{mytheo}
  \label{thm:Akvalue}
 Let $k - \ell = 12n + j$ with $\ell \geq 14$, $n \geq 1$ and $0 \leq j <12$.  Then $A_{k,\ell} \geq N_{k,\ell}'$ where $N_{k,\ell}'$ is given by the following table of values:
 \begin{equation}
 \begin{array}{c|c|c|c}
j \thinspace \diagdown \thinspace \ell \mymod{6} & 0 & 2 & 4 \\
\hline
0 & n & n & n-1 \\
\hline
2 & n-1 & n & n-1  \\
\hline
4 & n & n & n \\
\hline
6 & n & n & n-1  \\
\hline
8 & n & n+1 & n  \\
\hline
10 & n & n & n
\end{array}.
\end{equation}
\end{mytheo}
To compare the lower bound in Theorem \ref{thm:Akvalue} with some data, see Table \ref{table:aklChartdifference}.  
What is apparently true from the table (and indeed is proven in Corollary \ref{coro:allzerosonboundary} below) is that for $k$ sufficiently large compared to $\ell$, $A_{k,\ell} = N_{k,\ell}'$. Note that for $\ell=24$, this stabilization in the difference between $A_{k,\ell}$ and $N_{k,\ell}'$ has already occured by $k=56$, while for $\ell=22$, such a stabilization does not occur until $\ell=78$. 

\begin{table}[]
\centering
\caption{Values of $A_{k,\ell}-N_{k,\ell}'$.  
Entries are underlined to emphasize the behavior 
when $k-\ell \equiv 0 \pmod{12}$ as $k$ is increased and $\ell$ is held constant.}
\label{table:aklChartdifference}
\begin{tabular}{l|l|l|l|l|l|l|l|l|l|l|l|l|l|l|l}
$\ell \thinspace \diagdown \thinspace k$ & 56 & 58 & 60 & 62 & 64 & 66 & 68 & 70 & 72 & 74 & 76 & 78 & 80 & 82 & 84 \\
\hline
20  & \underline{0}  & 0  & 0  & 0  & 0  & 0  & \underline{0}  & 0  & 0  & 0  & 0  & 0  & \underline{0}  & 0  & 0  \\
\hline
22  & 0  & \underline{1}  & 0  & 0  & 1  & 0  & 0  & \underline{1}  & 0  & 0  & 1  & 0  & 0  & \underline{0}  & 0  \\
\hline
24  & 0  & 0  & \underline{0}  & 0  & 0  & 0  & 0  & 0  & \underline{0}  & 0  & 0  & 0  & 0  & 0  & \underline{0} 
\end{tabular}
\end{table}

The corresponding behavior of $B_{k,\ell}$ is more complicated to state.  
The difficulty is that $k$ needs to be large enough for the value of $B_{k,\ell}$ to stabilize; in the case of $A_{k,\ell}$ the lower bound in Theorem \ref{thm:Akvalue} is true, but not always sharp, except if $k$ is large.  The size of $k$ required turns out to depend on $k-\ell \pmod{6}$, and on $\ell \pmod{6}$.  Towards this end, we define the following function.
\begin{mydefi}
Let $k-\ell = 12n + j$, with $0 \leq j < 12$.
Set $\stabp_j(\ell) = \ell$, except for the following cases:
\begin{equation}
\label{eq:stabpdef}
\stabp_{j}(\ell) = 
\begin{cases}
 \tfrac{\ell -1 + \sqrt{3 \ell^2 -1}}{2} \quad &\text{ if } \ell \equiv 0 \pmod{6}, \text{ and } j\equiv 2 \pmod{6},
\\
2 \ell , \quad &\text{ if } \ell \equiv 4 \pmod{6}, \text{ and } j \equiv 2 \pmod{6}, 
\\
\tfrac{4\ell -1 + \sqrt{12 \ell^2 -12 \ell + 1}}{2} \quad &\text{ if } \ell \equiv 4 \pmod{6}, \text{ and } j\equiv 0 \pmod{6}.
\end{cases}
\end{equation}
\end{mydefi}
Remark.  Here $\stabp$ stands for ``stablization point," and is intended to represent how large $k$ needs to be compared to $\ell$ for the values of $A_{k,\ell}$ and $B_{k,\ell}$ to stabilize to their limiting behaviors.

\begin{mytheo}
 \label{thm:Bkvalue}
 Let $k \geq \ell \geq L$, where $L$ is some large absolute (effective) constant.  Then $B_{k,\ell} \geq T_{k,\ell}'$, provided $k \geq \stabp_{j}(\ell)$, where
 \begin{equation}
  T_{k,\ell}' =  \begin{cases}
   \lfloor \ell/6 \rfloor - 1, \qquad &\ell \equiv 0 \mymod{6},  \\
   \lfloor \ell/6 \rfloor - 1, \qquad &\ell \equiv 2 \mymod{6}, \\
   \lfloor \ell/6 \rfloor , \qquad &\ell \equiv 4 \mymod{6}.
  \end{cases}
 \end{equation}
\end{mytheo}
As an example to illustrate Theorem \ref{thm:Bkvalue}, note that when $\ell = 22 \equiv 4 \pmod{6}$, Table \ref{table:bklChart} indicates that the value of $B_{k,\ell}$ perhaps stablizes at $k=82$ (but certainly not at $k=70$).  For comparison, we have $sp_{0}(22) = 81$, so indeed $82$ is where the value stablizes in this example.  As another example, numerical calculations indicate $B_{56,42} = 5$ and $B_{68,42} = 6$; here $\ell = 42 \equiv 0 \pmod{6}$, and $j=2$.  We have $\stabp_{2}(42) = 57$, which is in agreement with the theorem.


From Theorems \ref{thm:Akvalue} and \ref{thm:Bkvalue}, we deduce
\begin{mycoro}
\label{coro:AklPlusBklTableOfValues}
 Suppose $k \geq \ell \geq L$, $k-\ell =12n +j$ with $n\geq 1$, and $k \geq \stabp_{j}(\ell)$.  Then $A_{k,\ell} + B_{k,\ell}$ is bounded from below by the following table of values:
 \begin{equation}
 \label{eq:valenceTable}
\begin{array}{c|c|c|c}
j \diagdown \ell \mymod{6} & 0 & 2 & 4 \\
\hline
0 & n + \lfloor \frac{\ell}{6} \rfloor - 1 & n + \lfloor \frac{\ell}{6} \rfloor - 1 & n + \lfloor \frac{\ell}{6} \rfloor - 1 \\
2 & n + \lfloor \frac{\ell}{6} \rfloor - 2 & n + \lfloor \frac{\ell}{6} \rfloor - 1 & n + \lfloor \frac{\ell}{6} \rfloor - 1 \\
4 & n + \lfloor \frac{\ell}{6} \rfloor - 1 & n + \lfloor \frac{\ell}{6} \rfloor - 1 & n + \lfloor \frac{\ell}{6} \rfloor  \\
6 & n + \lfloor \frac{\ell}{6} \rfloor - 1 & n + \lfloor \frac{\ell}{6} \rfloor  -1 & n + \lfloor \frac{\ell}{6} \rfloor - 1 \\
8 & n + \lfloor \frac{\ell}{6} \rfloor - 1 & n + \lfloor \frac{\ell}{6} \rfloor  & n + \lfloor \frac{\ell}{6} \rfloor \\
10 & n + \lfloor \frac{\ell}{6} \rfloor - 1 & n + \lfloor \frac{\ell}{6} \rfloor - 1 & n + \lfloor \frac{\ell}{6} \rfloor.  \\
\end{array}
\end{equation}
\end{mycoro}

One can check by a tedious but straightforward calculation that \eqref{eq:valenceTable} agrees with $\lfloor \frac{k+\ell}{12} \rfloor - 1$, except when $k+ \ell \equiv 2 \pmod{12}$, in which case it equals $\lfloor \frac{k+\ell}{12} \rfloor - 2$.  This is precisely the number of nontrivial zeros of a cusp form of weight $k+\ell$.  It may be worth mentioning that this calculation does not require $n \geq 1$; the only possible caveat is that when the table gives a negative value, that implies $k + \ell \leq 14$, which are degenerate situations that we are avoiding anyway.

Therefore, we deduce 
\begin{mycoro}
\label{coro:allzerosonboundary}
  Suppose $k \geq \ell \geq L$, $k-\ell =12n +j$ with $n\geq 1$, and $k \geq \stabp_{j}(\ell)$.  Then all the zeros of $\Delta_{k,\ell}$ in the standard fundamental domain lie on the boundary.  Furthermore, the lower bounds on $A_{k,\ell}$ and $B_{k,\ell}$ given in Theorems \ref{thm:Akvalue} and \ref{thm:Bkvalue} are equalities, that is, $A_{k,\ell} = N_{k,\ell}'$ and $B_{k,\ell} = T_{k,\ell}'$.
\end{mycoro}

When $\ell \leq k < \stabp_j(\ell)$, then there is still one unaccounted for zero of $\Delta_{k,\ell}$.  For this, we have
\begin{myprop}
 Suppose that $L \leq \ell \leq k < \stabp_j(\ell)$, and $k-\ell =12n +j$ with $n \geq 1$.  Then $A_{k,\ell} \geq N_{k,\ell}' + 1$.
\end{myprop}
This provides the final missing zero, and shows
that if $\ell \geq L$, then all the zeros of $\Delta_{k,\ell}$ in $\mathcal{F}$ lie on the boundary, provided $k = 12n + j$ with $n \geq 1$.  The cases with $n=0$ are covered by Proposition \ref{prop:nozerosonarc}.  Altogether, the subsidiary results described above combine to give Theorem \ref{thm:mainthmAllZerosOnBoundary}.

Rudnick \cite[Theorem 2]{Rudnick} showed that if a family of cusp forms $f_k$ satisfies QUE\footnote{In a strong form which rules out ``escape of mass.''}, then the zeros of the $f_k$ become equidistributed in $\mathcal{F}$.  Our results on the zeros here then show that $\Delta_{k, \ell}$ do not satisfy QUE as $k+\ell \rightarrow \infty$.  It could be interesting to show directly that $\Delta_{k,\ell}$ have escape of mass (if true).

The assumption that $\ell \geq L$ in the above results is intended to simplify and shorten the proofs throughout this paper.  A value of $L$ could certainly be calculated with more work, but it is unlikely that $\ell=4$ (for instance) would be covered by this method.  The first-named author has shown for $\ell=4,6,8$ that $\Delta_{k,\ell}$ has all its zeros on the arc (we omitted the proof from this paper for brevity).   
For any particular value of $k$ and $\ell$, a computer can calculate $A_{k,\ell}$ and $B_{k,\ell}$, so it seems feasible that Conjecture \ref{conj:mainconj} could be proved using this strategy.

It is also interesting to study the finer distribution of the zeros of $\Delta_{k,\ell}$.  Our work in this paper shows that $\Delta_{k,\ell}$ has a zero on the arc in each interval of the form $(m \frac{2 \pi}{k-\ell}, (m+1) \frac{2 \pi}{k-\ell}) \subset (\frac{\pi}{3}, \frac{\pi}{2})$, provided $k-\ell \rightarrow \infty$, except possibly for a bounded number of zeros near $e^{i\pi/3}$.  From this, we may conclude that the zeros of $\Delta_{k,\ell}$ on the arc equidistribute on the arc uniformly with respect to $\theta$, as $k-\ell \rightarrow \infty$.  Similarly, $\Delta_{k,\ell}$ has a zero on the line $x=1/2$ between the points $1/2 + iy_m$ and $1/2 + i y_{m+1}$, where $1/2 + i y_m = R e^{i \theta_m}$, with $\theta_m = \frac{\pi m}{\ell} \in (\frac{\pi}{3}, \frac{\pi}{2})$, provided $\ell$ is large.  Thus the zeros of $\Delta_{k,\ell}$ on the line $x=1/2$ equidistribute according to their argument, as $\ell \rightarrow \infty$.

\subsection{Overview}
\label{section:Overview}
For context, we summarize the method of Rankin and Swinnerton-Dyer \cite{RS} showing that all the zeros of $E_k$ lie on the arc.
Let $F_k(\theta) = e^{ik \theta} E_k(e^{i \theta})$, so that by \eqref{eq:EkczdDefinition}, we have
\begin{equation}
 F_k(\theta) = \tfrac{1}{2} \sum_{(c,d) = 1} (c e^{i\theta/2} + d e^{-i\theta/2})^{-k}.
\end{equation}
Now $F_k$ is real-valued and 
vanishes if and only if $E_k(e^{i \theta})$ vanishes.  Explicitly evaluating  the terms with $c^2 + d^2 \leq 2$, and quoting a bound from \cite{RS} for the terms with $c^2 + d^2 \geq 5$, we derive
\begin{equation}
\label{eq:FkApproximationInitial}
 F_k(\theta) = 2\cos(\tfrac{k \theta}{2}) + (2 \cos(\tfrac{\theta}{2}))^{-k} + (2i \sin(\tfrac{\theta}{2}))^{-k} + R_k(\theta),
\end{equation} 
where
\begin{equation}
\label{eq:Rkthetabound}
 |R_k(\theta)| \leq 4(\tfrac{5}{2})^{-\frac{k}{2}} + \tfrac{20\sqrt{2}}{k-3}(\tfrac{9}{2})^{\frac{3-k}{2}}.
\end{equation}
Note that if $k \geq 14$, then $|R_k(\theta)| \leq 7.3 \times 10^{-3}$.  Moreover, for $\theta \in [\frac{\pi}{3}, \frac{\pi}{2}]$, $2\cos(\frac{\theta}{2}) \geq \sqrt{2}$, and $2\sin(\frac{\theta}{2}) \geq 1$, so $|F_k(\theta) - 2 \cos(\frac{k \theta}{2})| \leq 1.016$.  By choosing values of $\theta$ so that $2 \cos(\frac{k \theta}{2}) = \pm 2$, we see that $F_k(\theta)$ alternates sign along these points, and therefore changes sign between these points.  A careful count of these points, and a comparison to the valence formula \eqref{eq:valenceFormula}, shows that this accounts for all the zeros of $E_k$.

Our work here follows the same general method of approximating $\Delta_{k,\ell}$ and evaluating it at a collection of sample points.  Our approximation is more complicated because it is necessary to retain more terms, as indicated in \eqref{eq:FkApproximationInitial}.  The analysis of the sample points along the arc becomes most intricate near the corner point $e^{i \pi/3}$.

We also separately perform an approximtion for $\Delta_{k,\ell}$ along the side $x=1/2$.  For small values of $y$ (which turns out to be the rather wide range $1\ll y \ll k^{1/2 + o(1)})$, we can safely discard the terms in \eqref{eq:EkczdDefinition} with $c^2 + d^2 \geq 5$, similarly to \eqref{eq:FkApproximationInitial}.  Meanwhile, for $y \gg k^{1/2 + o(1)}$, it is easiest to apply the Fourier expansion, which recall 
takes the form
\begin{equation}
\label{eq:EkFourierExpansion}
 E_k(z) = 1 + \gamma_k \sum_{n=1}^{\infty} \sigma_{k-1}(n) e(nz), 
\end{equation}
where with $B_k$ denoting the Bernoulli numbers,
\begin{equation}
 \gamma_k = (-1)^{k/2} \frac{(2 \pi )^k}{\Gamma(k) \zeta(k)} = \frac{-2k}{B_k}.
\end{equation}
It turns out that the Fourier expansion is well-approximated by a single term $n \approx \frac{k}{2 \pi y}$, in this range.  The problem now is that the range with $y \asymp k^{1/2}$ is not covered by these two types of approximations.  We have been able to cover this range by the approximation $E_k(z) = 1 + \sum_{d \in \mz} (z+d)^{-k} + \sum_{|c| \geq 2} (\dots)$, with the sum over $|c| \geq 2$ being a small error term.  Then we develop a more precise approximation for $\sum_{d \in \mz} (z+d)^{-k}$ in terms of the Jacobi theta function, which recall is defined by
by
\begin{equation}
\label{eq:JacobithetaDefinition}
\theta(w, \tau) = \sum_{n \in \mz} \exp(\pi i n^2 \tau + 2 \pi i n w),
\end{equation}
for $\tau \in \mathbb{H}$ and $w \in \mc$.  This feature of the Eisenstein series is of independent interest, and should be useful for other investigations into the behavior of large weight Eisenstein series.  See Theorem \ref{thm:EisensteinApproxTheorem} below for this result.

The Jacobi theta function satisfies the modularity relation
\begin{equation}
\label{eq:JacobithetaModularity}
 \theta(w/\tau, -1/\tau) = (-i \tau)^{1/2} \exp(\pi i w^2/\tau) \theta(w,\tau).
\end{equation}
The use of the modularity relation turns out to be crucial for approximations throughout the transition region $y \asymp k^{1/2}$.
Of course, modularity of the theta function, and the Fourier expansion of the Eisenstein series, are both applications of the Poisson summation formula, so it is no surprise that these are reflections of this fact.

\section{Zeros on the arc}
\subsection{Estimates for trigonometric functions}
We first introduce some lemmas.
\begin{mylemma}
\label{lemma:TrigEasyEstimates}
Suppose that $N \geq 8$, $N \equiv 2$ (mod $6$), and $\tfrac{\pi}{3} \leq \theta \leq \tfrac{\pi}{3}(1+\tfrac{1}{N})$. Then $\vert \cos(\tfrac{N\theta}{2}) \vert \leq \tfrac{1}{2}$.  Similarly, if $N \geq 10$, $N \equiv 4$ (mod $6$), and $\tfrac{\pi}{3} \leq \theta \leq \tfrac{\pi}{3}(1+\tfrac{1}{2N})$, then $\tfrac{1}{2} \leq \vert \cos(\tfrac{N\theta}{2}) \vert \leq \frac{1}{\sqrt{2}}$.
\end{mylemma}
\begin{proof}
We have
\begin{equation}
\cos(\tfrac{N\theta}{2}) 
= \cos(\tfrac{N\pi}{6}+\tfrac{N}{2}(\theta-\tfrac{\pi}{3})).
\end{equation}
In the case $N \equiv 2 \pmod{6}$, the assumed bound on $\theta$ implies 
$0 \leq \tfrac{N}{2}(\theta-\tfrac{\pi}{3}) \leq \tfrac{\pi}{6}$. 
If $N \equiv 2 \pmod{6}$, then modulo $\pi$, we have
$\tfrac{N\theta}{2}  \in [\tfrac{\pi}{3}, \tfrac{\pi}{2}]$, and therefore $\vert \cos(\tfrac{N\theta}{2}) \vert \leq \tfrac{1}{2}$.  A similar argument applies to $N \equiv 4 \pmod{6}$, giving that modulo $\pi$, we have $\frac{N \theta}{2} \in [\frac{2\pi}{3}, \frac{3\pi}{4}]$. 
\end{proof}

\begin{mylemma}
\label{lemma:sinMonotonic}
The function $(2 \sin(\frac{\pi}{6} + \frac{\pi}{X}))^{-X}$ is monotonically decreasing for $X > 6/5$.  
\end{mylemma}
\begin{proof}
It suffices to show that $f(X):= -X \log(2 \sin(\frac{\pi}{6} + \frac{\pi}{X}))$ is decreasing, which amounts to showing $f'(X) \leq 0$.  We have
\begin{equation}
 f'(X) = - \log(2 \sin(\tfrac{\pi}{6} + \tfrac{\pi}{X})) + \tfrac{\pi}{X} \cot(\tfrac{\pi}{6} + \tfrac{\pi}{X}).
\end{equation}
Changing variables $Y = \frac{\pi}{X}$, we then want to show that
\begin{equation}
 g(Y) := - \log(2 \sin(\tfrac{\pi}{6} + Y) + Y \cot(\tfrac{\pi}{6} + Y) \leq 0,
\end{equation}
for $0 \leq Y < \frac{5 \pi}{6}$.  For this, we note $g(0) = 0$, and
\begin{equation}
 g'(Y) = - Y \csc(\tfrac{\pi}{6} + Y)^2,
\end{equation}
which is $\leq 0$ on the interval $[0, 5 \pi /6)$.  
\end{proof}
\begin{mycoro}
\label{coro:sinBounds}
Suppose $\frac{\pi}{3} + \frac{\pi}{N} \leq \theta \leq \frac{\pi}{2}$, where $N \geq 6$ is real.  Then $(2 \sin(\tfrac{\theta}{2}))^{-N} \leq \frac18$.  
\end{mycoro}
\begin{proof}
By elementary considerations, the minimal value of
$(2 \sin(\theta/2))^{-N}$ occurs at $\theta = \frac{\pi}{3} + \frac{\pi}{N}$.  
By Lemma \ref{lemma:sinMonotonic}, 
\begin{equation}
\label{eq:sinBoundansatz}
(2 \sin(\tfrac{\pi}{6} + \tfrac{\pi}{2N}))^{-N} = \Big[(2 \sin(\tfrac{\pi}{6} + \tfrac{\pi}{2N}))^{-2N}\Big]^{1/2} \leq (2 \sin(\tfrac{\pi}{6} + \tfrac{\pi}{12}))^{-6} = \tfrac{1}{8}. \qedhere
\end{equation}
\end{proof}
Remark.  In the forthcoming work, we shall apply variations on Corollary \ref{coro:sinBounds}, which all follow the same steps indicated in \eqref{eq:sinBoundansatz}.  We have found it easiest to simply repeat the argument on demand.

\subsection{Estimates for the Eisenstein series on the arc}
Similarly to the method of Rankin and Swinnerton-Dyer discussed in Section \ref{section:Overview}, we shall work with $F_k F_{\ell} - F_{k+\ell}$, which is real-valued and vanishes if and only if $\Delta_{k,\ell}$ vanishes.

\begin{mylemma}
\label{lemma:FkapproxOnArc}
 For $\theta \in [\frac{\pi}{3}, \frac{\pi}{2}]$, and for $k \geq \ell \geq 14$, we have
\begin{equation}
 (F_k F_{\ell} - F_{k+\ell})(\theta) = 
 M_{k,\ell}(\theta) + 
  \mathcal{E}_{k,l}(\theta),
\end{equation}
where
\begin{equation}
\label{eq:MkellDef}
 M_{k,\ell}(\theta) = 2\cos(\tfrac{(k-\ell)\theta}{2}) +2\cos(\tfrac{k \theta}{2})(2i\sin(\tfrac{\theta}{2}))^{-\ell} + 2\cos(\tfrac{\ell \theta}{2})(2i\sin(\tfrac{\theta}{2}))^{-k},
\end{equation}
and
\begin{equation}
\label{eq:mathcalEerror}
 |\mathcal{E}_{k,\ell}(\theta)| \leq 0.091.
\end{equation}
\end{mylemma}
Remark.  It is necessary to include the secondary terms in \eqref{eq:MkellDef} in order to approximate $F_k F_{\ell} - F_{k+\ell}$ with $\theta$ near $\frac{\pi}{3}$.
\begin{proof}
By direct calculation with \eqref{eq:FkApproximationInitial}, using $4\cos( \frac{k\theta}{2}) \cos( \frac{\ell\theta}{2}) - 2 \cos(\frac{(k + \ell) \theta}{2}) = 2 \cos(\frac{(k-\ell)\theta}{2})$, we have
\begin{multline}
\label{eq:FkFellformula}
(F_kF_{\ell}-F_{k+\ell})(\theta) = 2\cos(\tfrac{(k-\ell)\theta}{2}) +2\cos(\tfrac{k \theta}{2})(2i\sin(\tfrac{\theta}{2}))^{-\ell} + 2\cos(\tfrac{\ell \theta}{2})(2i\sin(\tfrac{\theta}{2}))^{-k} 
\\ 
+ 2\cos(\tfrac{k\theta}{2})(2\cos(\tfrac{\theta}{2}))^{-\ell} + 2\cos(\tfrac{\ell \theta}{2})(2\cos(\tfrac{\theta}{2}))^{-k} 
\\
+(2\cos(\tfrac{\theta}{2}))^{-k}(2i\sin(\tfrac{\theta}{2}))^{-\ell} 
+(2\cos(\tfrac{\theta}{2}))^{-\ell}(2i\sin(\tfrac{\theta}{2}))^{-k} + R_{k,\ell}(\theta),
\end{multline}
where 
\begin{multline}
 |R_{k,\ell}(\theta)| \leq |R_k(\theta)| (2+(2\cos(\tfrac{\theta}{2}))^{-\ell}+(2\sin(\tfrac{\theta}{2}))^{-\ell})  
 \\
 + |R_{\ell}(\theta)| (2+(2\cos(\tfrac{\theta}{2}))^{-k}+(2\sin(\tfrac{\theta}{2}))^{-k}) 
 + |R_k(\theta) R_{\ell}(\theta)| + |R_{k+\ell}(\theta)|.
\end{multline}
%
Using that $(2 \cos(\frac{\theta}{2}))^{-1} \leq 2^{-1/2}$  and $(2 \sin(\frac{\theta}{2}))^{-1} \leq 1$ on $[\frac{\pi}{3}, \frac{\pi}{2}]$, and combining with \eqref{eq:Rkthetabound}, we then easily obtain a bound on $|R_{k,\ell}(\theta)|$ that is decreasing in both $k$, $\ell$.  Evaluating the bound at $k=\ell = 14$, we derive that
\begin{equation}
\label{eq:Rkellbound}
 |R_{k,\ell}(\theta)| \leq 0.044.
\end{equation}

Similarly, we bound the trigonometric terms on the second and third lines of \eqref{eq:FkFellformula} by $2 \times 2^{-\ell/2} + 2 \times 2^{-k/2} + 2^{-k/2} + 2^{-\ell/2} \leq 6 \times 2^{-7} = 0.046\dots$.  Adding this with \eqref{eq:Rkellbound} gives the desired result.
\end{proof}

\subsection{Sample points}
\label{section:samplepointsAkl}
Our goal now is to count sign changes of $F_k F_{\ell} - F_{k+\ell}$ by finding a sequence of points where this function changes sign.  
Write $k-\ell = 12n+j$ for $n \geq 1$  and $j \in \{0,2,4,6,8,10\}$. We claim that $F_kF_{\ell} - F_{k+\ell}$ exhibits sign changes at the points $\theta_m = \tfrac{2m\pi}{k-\ell} = \tfrac{2m\pi}{12n+j}$ where $m$ ranges over integers such that $\theta_m \in (\frac{\pi}{3}, \frac{\pi}{2}]$ (equivalently, $m \in (2n + \frac{j}{6}, 3n + \frac{j}{4}]$).
By an elementary calculation, there are $n$ integers in this interval if $j=0,2,6$ and $n+1$ if $j=4,8,10$.

We have
\begin{equation}
 2\cos(\tfrac{(k-\ell)\theta_m}{2}) = 2\cos(m \pi) = 2(-1)^m, \quad \text{and} \quad
 2\cos(\tfrac{k\theta_m}{2}) = 2(-1)^m\cos(\tfrac{\ell \theta_m}{2}).
\end{equation}
%
%
Therefore,
\begin{equation}
\label{eq:Mellnlesspsecific}
M_{k, \ell}(\theta_m) = 2(-1)^m [1 + \cos(\tfrac{\ell \theta_m}{2}) \{(2i \sin(\tfrac{\theta_m}{2}))^{-\ell} + (-1)^m (2i \sin(\tfrac{\theta_m}{2}))^{-k}\}].
\end{equation}
%
%
Write $m = 2n + r$, with $r \in (\frac{j}{6}, n + \frac{j}{4}]$, 
and define 
\begin{equation}
x = \frac{\pi(r-\frac{j}{6})}{12n + j}.  
\end{equation}
Then with this notation, we have
\begin{equation}
 \theta_m = \frac{2 \pi (2n+r)}{12n+j}  = \frac{\pi}{3} + 2x,
\end{equation}
and so
\begin{equation}
\label{eq:Mellnthetamformula}
M_{k, \ell}(\theta_m) = 2 (-1)^r [1 + \cos(\tfrac{\ell \pi}{6} + \ell x)(2 i \sin(\tfrac{\pi}{6} + x))^{-\ell} \{1 + (-1)^r (2 i \sin(\tfrac{\pi}{6} + x))^{-12n-j} \} ].
\end{equation}

\begin{myprop}
\label{prop:Mksamplepoints}
 For $\theta_m \in (\frac{\pi}{3}, \frac{\pi}{2}]$ and $k \geq \ell \geq 14$, we have
 \begin{equation}
(-1)^m M_{k,\ell}(\theta_m) \geq 
\begin{cases}
1.5, \qquad &\ell \equiv 0 \pmod{6} \\
0.8, \qquad &\ell \equiv 2 \pmod{6} \\
0.31, \qquad & \ell \equiv 4 \pmod{6}.
\end{cases}
\end{equation}
\end{myprop}
\begin{proof}
First assume $\ell \equiv 0 \pmod{6}$.  In this case, \eqref{eq:Mellnthetamformula} simplifies as
\begin{equation}
\label{eq:Mellnthetamformulal=0mod6}
(-1)^r M_{k,\ell}(\theta_m) = 2 [1 + \cos( \ell x)( 2\sin(\tfrac{\pi}{6} + x))^{-\ell} \{1 + (-1)^r (2 i \sin(\tfrac{\pi}{6} + x))^{-12n-j} \} ],
\end{equation}
using that $\cos(\frac{\ell \pi}{6} + \ell x) i^{-\ell} = \cos(\ell x)$ for $\ell \equiv 0 \pmod{6}$. 
In case $x \leq \frac{\pi}{2\ell}$, then $0 \leq \ell x \leq \frac{\pi}{2}$, so $\cos(\ell x) \geq 0$.  The term in curly brackets lies in the interval $[0,2]$, so in all we have
$(-1)^r M_{k,\ell}(\theta_m) \geq 2.$  
In case $x > \frac{\pi}{2 \ell}$, then by Corollary \ref{coro:sinBounds}, $(2 \sin(\frac{\pi}{6} + x))^{-\ell} \leq \frac{1}{8}$.  Therefore, 
\begin{equation}
(-1)^r M_{k,\ell}(\theta_m) \geq 2(1-2 (\tfrac{1}{8})) = 1.5.
\end{equation}

Now suppose $\ell \equiv 2 \pmod{6}$. In this case, \eqref{eq:Mellnthetamformula} simplifies as
\begin{equation}
\label{eq:Mellnthetamformulal=2mod6}
(-1)^r M_{k,\ell}(\theta_m) = 2 [1 - \cos( \tfrac{\pi}{3} + \ell x)(2\sin(\tfrac{\pi}{6} + x))^{-\ell} \{1 + (-1)^r (2 i \sin(\tfrac{\pi}{6} + x))^{-12n-j} \} ],
\end{equation}
using that $\cos(\frac{\ell \pi}{6} + \ell x) i^{-\ell} = -\cos(\frac{\pi}{3}+\ell x)$ for $\ell \equiv 2 \pmod{6}$. 
First, suppose $\tfrac{\pi}{3} \leq \theta_m \leq \tfrac{\pi}{3}(1+\tfrac{1}{\ell})$. By Lemma \ref{lemma:TrigEasyEstimates}, $| \cos(\tfrac{\pi}{3}+ \ell x) | \leq \tfrac{1}{2}$. 
Furthermore, we claim that 
\begin{equation}
(2 \sin(\tfrac{\pi}{6} + x))^{-12n -j }
 \leq 0.193,
\end{equation}
as we now show.
Since $r > \frac{j}{6}$, we have $r - \frac{j}{6} \geq \frac{1}{3}$, and so by Lemma \ref{lemma:sinMonotonic}, we have
\begin{equation}
\label{eq:sinSystemofInequalities}
 (2 \sin(\tfrac{\pi}{6} + \tfrac{\pi(r-\frac{j}{6})}{12n+j}))^{-12n-j} \leq (2 \sin(\tfrac{\pi}{6}+\tfrac{\pi/3}{12n+j}))^{-12n-j} \leq (2 \sin(\tfrac{\pi}{6}+\tfrac{\pi/3}{12}))^{-12} \leq 0.192\dots,
\end{equation}
as claimed.  It is useful to note that the above proof of \eqref{eq:sinSystemofInequalities} did not use that $\ell \equiv 2 \pmod{6}$, so we may use this bound again for the other residue classes of $\ell$.
Inserting these bounds into \eqref{eq:Mellnthetamformulal=2mod6}, we derive
\begin{equation}
\label{eq:Mkellthetamboundl=2mod6}
 (-1)^r M_{k,\ell}(\theta_m) \geq 2[1- \tfrac12 \{1 + 0.193 \}] \geq 0.8.
\end{equation}
We finish the case $\ell \equiv 2 \pmod{6}$ by considering the values of $m$ such that
$\tfrac{\pi}{3}(1+ \tfrac{1}{\ell}) \leq \theta_m \leq \tfrac{\pi}{2}$.  The bound \eqref{eq:sinSystemofInequalities} remains valid.  Lemma \ref{lemma:sinMonotonic} gives
\begin{equation}
(2 \sin(\tfrac{\pi}{6} + x))^{-\ell}
 \leq (2 \sin(\tfrac{\pi}{6} + \tfrac{\pi}{6 \ell}))^{-\ell} \leq (2 \sin(\tfrac{\pi}{6} + \tfrac{\pi}{48}))^{-8} = 0.43\dots,
\end{equation}
which leads to an even better bound than \eqref{eq:Mkellthetamboundl=2mod6}.

Finally, consider the case $\ell \equiv 4 \pmod{6}$.
In this case, \eqref{eq:Mellnthetamformula} simplifies as
\begin{equation}
M_{k,\ell}(\theta_m) = 2(-1)^r[1+\cos(\tfrac{2\pi}{3}+\ell x)(2\sin(\tfrac{\pi}{6}+x))^{-\ell} \{1 + (-1)^r (2 i \sin(\tfrac{\pi}{6} + x))^{-12n-j} \} ],
\end{equation}
using that $\cos(\frac{\ell \pi}{6} + \ell x)i^{-\ell} = \cos(\frac{2\pi}{3}+\ell x)$ for $\ell \equiv 4 \pmod{6}$.
First, suppose $\tfrac{\pi}{3} \leq \theta_m \leq \tfrac{\pi}{3}(1+\tfrac{1}{2\ell})$. 
Then by Lemma \ref{lemma:TrigEasyEstimates}, $\vert \cos(\tfrac{2\pi}{3}+\ell x) \vert \leq \frac{1}{\sqrt{2}}$. 
By 
\eqref{eq:sinSystemofInequalities},
in all, we have 
\begin{equation}
\label{eq:Mkelllowerbound}
(-1)^r M_{k,\ell}(\theta_m) \geq 2[1 - \tfrac{1}{\sqrt{2}}\{1+0.193\}] = 0.31\dots.
\end{equation}
Now suppose $\tfrac{\pi}{3}(1+\tfrac{1}{2\ell}) \leq \theta_m \leq \tfrac{\pi}{2}$. Then by Lemma \ref{lemma:sinMonotonic}, we have
\begin{equation}
(2\sin(\tfrac{\theta_m}{2}))^{-\ell} 
\leq (2\sin(\tfrac{\pi}{6} + \tfrac{\pi}{12 \ell}))^{-\ell} 
\leq (2\sin(\tfrac{\pi}{6} + \tfrac{\pi}{48}))^{-4}
\leq 0.656\dots < 2^{-1/2}.
\end{equation}
Therefore, for $\theta_m$ in this range we have an even better bound than \eqref{eq:Mkelllowerbound}.
\end{proof}

Since the error term in \eqref{eq:mathcalEerror} is smaller than the lower bound in Proposition \ref{prop:Mksamplepoints}, we may definitively conclude that $F_k F_{\ell} - F_{k+\ell}$ alternates signs at the values of $\theta_m \in (\frac{\pi}{3}, \frac{\pi}{2}]$.  By counting the number of such sample points, as described in Section \ref{section:samplepointsAkl}, we finish the proof of Proposition \ref{prop:Akllowerbound}.
%
%
\subsection{The extra zero}
Our task in this section is to prove Theorem \ref{thm:Akvalue}.  By comparison with Proposition \ref{prop:Akllowerbound}, we need to produce one additional zero when $\ell \equiv 0 \pmod{6}$ and $j=0,6$, and when $\ell \equiv 2 \pmod{6}$ and $j=0,2,6,8$.

\begin{myprop}
\label{prop:ExtraZerol=0mod6}
 Suppose $\ell \equiv 0 \pmod{6}$, $\ell \geq 14$, and $k - \ell = 12n + j$, with $n \geq 1$.  For $j=0,6$, we have that $A_{k,\ell} \geq n$.
\end{myprop}

\begin{myprop}
\label{prop:ExtraZerol=2mod6}
 Suppose $\ell \equiv 2 \pmod{6}$, $\ell \geq 14$, and $k - \ell = 12n + j$, with $n \geq 1$.  For $j=0,2,6$, we have that $A_{k,\ell} \geq n$.  In addition, for $j=8$, we have that $A_{k,\ell} \geq n+1$.
\end{myprop}
These two propositions provide the extra zero not counted by Proposition \ref{prop:Akllowerbound}.  Notice that there is essentially no condition on the size of $k$, although it is required that $\ell \geq 14$.


\begin{proof}[Proof of Proposition \ref{prop:ExtraZerol=0mod6}]
We will show there is one additional sign change near $\theta = \pi/3$.

First, suppose $j=0$, so $k-\ell = 12n$.  
Consulting the proof of Proposition \ref{prop:Mksamplepoints}, we see that the sample point closest to $\pi/3$ is $\theta_m$ with $m=2n+1$, in which case Proposition \ref{prop:Mksamplepoints} gives $M_{k,\ell}(\theta_{2n+1}) \leq -1.5$.  Meanwhile, directly evaluating $M_{k,\ell}(\pi/3)$ using
 \eqref{eq:MkellDef}, we have
\begin{equation}
M_{k,\ell}(\tfrac{\pi}{3}) = 2\cos(2n\pi)+2\cos(2n\pi+\tfrac{\ell\pi}{6})(2i\sin(\tfrac{\pi}{6}))^{-\ell} + 2\cos(\tfrac{\ell\pi}{6})(2i\sin(\tfrac{\pi}{6}))^{-12n-\ell} = 6.
\end{equation}
Therefore $F_kF_{\ell}-F_{k + \ell}$ changes sign at least one more time between $\tfrac{\pi}{3}$ and $\theta_{2n+1} = \tfrac{\pi}{3} + \tfrac{\pi}{6n}$.

A similar argument works for $j=6$. In this case, the sample point nearest to $\pi/3$ is $m=2n+2$, so $M_{k,\ell}(\theta_{2n+2}) \geq 1.5$.  On the other hand, direct evaluation shows $M_{k,\ell}(\frac{\pi}{3}) = -6$, and so there is one more sign change in this case also.
%
\end{proof}

\begin{proof}[Proof of Proposition \ref{prop:ExtraZerol=2mod6}]
First, suppose $j=0$ or $6$.  Recall from Section \ref{section:samplepointsAkl} that when $j=0$, the sample point nearest to $\pi/3$ is $\theta_{2n+1} = \frac{\pi}{3} + \frac{\pi}{6n}$, giving $M_{k,\ell}(\theta_{2n+1}) \leq -0.8$.  For $j=6$, the nearest point is  $\theta_{2n+2} = \frac{\pi}{3} + \frac{\pi}{6n+3}$, and we have $M_{k,\ell}(\theta_{2n+2}) \geq 0.8$. 

Let $\phi = \tfrac{\pi}{3} + \tfrac{\pi}{4(12n+j)}$.  We have
\begin{equation}
2\cos(\tfrac{(k-\ell)\phi}{2}) 
= 2\cos(\tfrac{\pi j}{6} + \tfrac{\pi}{8}) = 2 (-1)^{j/6} \cos( \tfrac{\pi}{8}),
\end{equation}
\begin{equation}
2\cos(\tfrac{k\phi}{2}) 
= 2(-1)^{j/6} \cos(\tfrac{\pi}{8}+ \tfrac{\ell \pi}{6} + \tfrac{\ell \pi}{8(12n+j)}), 
\end{equation}
and
\begin{equation}
2\cos(\tfrac{\ell \phi}{2}) = 2 \cos(\tfrac{\ell \pi}{6} + \tfrac{\ell \pi}{8(12n+j)}).
\end{equation}
By \eqref{eq:MkellDef}, we have
\begin{multline}
 |M_{k,\ell}(\phi) - 2 (-1)^{j/6} \cos(\tfrac{\pi}{8}) | \leq |2\cos(\tfrac{\pi}{8}+ \tfrac{\ell \pi}{6} + \tfrac{\ell \pi}{8(12n+j)})| (2 \sin(\tfrac{\phi}{2}))^{-\ell}
 \\
 + |2 \cos(\tfrac{\ell \pi}{6} + \tfrac{\ell \pi}{8(12n+j)})| 
 (2 \sin(\tfrac{\phi}{2}))^{-k}.
\end{multline}
As usual, we further subdivide into cases.  If $\tfrac{\pi}{3} \leq \phi \leq \tfrac{\pi}{3}(1+\tfrac{1}{\ell})$, or equivalently, 
$\ell \leq \frac43(12n + j)$,
then Lemma \ref{lemma:TrigEasyEstimates} implies $|\cos(\tfrac{\ell \phi}{2})| \leq \frac12$.  Furthermore, by Lemma \ref{lemma:sinMonotonic}, we have
\begin{equation}
 (2 \sin(\tfrac{\phi}{2}))^{-k} \leq (2 \sin(\tfrac{\pi}{6} + \tfrac{\pi}{8(12n+j)}))^{-12n-j} \leq (2 \sin(\tfrac{\pi}{6} + \tfrac{\pi}{96}))^{-12} = 0.519\dots.
\end{equation}
We also have that
\begin{equation}
 |\cos(\tfrac{k \phi}{2})| = |\cos(\tfrac{11 \pi}{24} + \tfrac{\ell \pi}{8(12n+j)})| \leq \sin(\tfrac{\pi}{8}) = 0.382\dots,
\end{equation}
which used $\frac{\ell \pi}{6} \equiv \frac{\pi}{3} \pmod{\pi}$, and $\frac{\ell \pi}{8(12n+j)} \leq \frac{\pi}{6}$.  Therefore, we conclude that
\begin{equation}
 |M_{k,\ell}(\phi) - 2 \cos(\tfrac{\pi}{8})| \leq 2 (0.383) + 0.520 \leq 1.286 \dots.
\end{equation}
In the complementary case where $\phi \geq \frac{\pi}{3}(1 + \frac{1}{\ell})$, equivalently, $\ell \geq \frac43(12n+j)$, then by Lemma \ref{lemma:sinMonotonic}, we have
\begin{equation}
 (2 \sin(\tfrac{\phi}{2}))^{-\ell} = [(2 \sin(\tfrac{\pi}{6} + \tfrac{\pi}{8(12n+j)}))^{-12n+j}]^{\frac{\ell}{12n+j}} \leq [(2 \sin(\tfrac{\pi}{6} + \tfrac{\pi}{96}))^{-12}]^{\frac43} = 0.417\dots.
\end{equation}
We then derive
\begin{equation}
  |M_{k,\ell}(\phi) - 2 \cos(\tfrac{\pi}{8})| \leq 4 (0.417 \dots) = 1.66 \dots.
\end{equation}
Since $2 \cos(\tfrac{\pi}{8}) = 1.847\dots$, we conclude that $M_{k,\ell}(\phi) \geq 1.84 - 1.67  = 0.17 $, which by comparison to \eqref{eq:mathcalEerror} is large enough to conclude that $F_k F_{\ell} - F_{k+\ell} > 0$ at $\phi$.

Finally, suppose $j=2$, so $k - \ell = 12n +2$.  The sample point closest to $\pi/3$ is $\theta_{2n+1}$, whereby $M_{k,\ell}(\theta_{2n+1}) \leq -0.8$.
By evaluating \eqref{eq:MkellDef} at $\pi/3$, we have
\begin{equation}
M_{k,\ell}(\tfrac{\pi}{3}) = 1 - 2\cos(\tfrac{2\pi}{3})(i^{\ell})(i^{-\ell}) +2\cos(\tfrac{\pi}{3})(i^{\ell})(i^{-\ell}) = 3,
\end{equation}
so this gives us one additional sign change.  If $j=8$, the calculation is similar to $j=2$, except the nearest sample point is $\theta_{2n+2}$, whereby $M_{k,\ell}(\theta_{2n+2}) \geq 0.8$, but then $M_{k,\ell}(\frac{\pi}{3}) = -3$, so the conclusion is the same as for $j=2$.
\end{proof}

\subsection{The case $n=0$}
When $n=0$, then Proposition \ref{prop:nozerosonarc} claims that $A_{k,\ell} = 0$ unless $j=8$, in which case $A_{k,\ell} = 1$.  Happily, all we need to prove about $A_{k,\ell}$ is that $A_{k,\ell} \geq 1$ when $k-\ell = 8$.  The necessary results on $B_{k,\ell}$ required for Proposition \ref{prop:nozerosonarc} appear in Section \ref{section:Bkl} (these results are not sensitive to the value of $k-\ell$, so there is no special consideration of the case $n=0$ for $B_{k,\ell}$).

Taking this for granted for a moment, let us pause to see how this implies Proposition \ref{prop:nozerosonarc}.  Using the lower bound from Theorem \ref{thm:Bkvalue} in the cases where $\stabp_j(\ell) = \ell$, which are all the cases but those given in \eqref{eq:stabpdef}, and otherwise using Proposition \ref{prop:BklLowerBoundAsymptotic}, we obtain the following table of lower bounds on $B_{k,\ell}$:
\begin{equation}
 \label{eq:BklTable}
\begin{array}{c|c|c|c}
j \diagdown \ell \mymod{6} & 0 & 2 & 4 \\
\hline
0 & \lfloor \frac{\ell}{6} \rfloor - 1 & \lfloor \frac{\ell}{6} \rfloor - 1 & \lfloor \frac{\ell}{6} \rfloor - 1 \\
2 &  \lfloor \frac{\ell}{6} \rfloor - 2 &  \lfloor \frac{\ell}{6} \rfloor - 1 &  \lfloor \frac{\ell}{6} \rfloor - 1 \\
4 &  \lfloor \frac{\ell}{6} \rfloor - 1 &  \lfloor \frac{\ell}{6} \rfloor - 1 &  \lfloor \frac{\ell}{6} \rfloor  \\
6 &  \lfloor \frac{\ell}{6} \rfloor - 1 &  \lfloor \frac{\ell}{6} \rfloor  -1 &  \lfloor \frac{\ell}{6} \rfloor - 1 \\
8 &  \lfloor \frac{\ell}{6} \rfloor - 2 &  \lfloor \frac{\ell}{6} \rfloor -1  &  \lfloor \frac{\ell}{6} \rfloor -1 \\
10 &  \lfloor \frac{\ell}{6} \rfloor - 1 &  \lfloor \frac{\ell}{6} \rfloor - 1 &  \lfloor \frac{\ell}{6} \rfloor
\end{array}.
\end{equation}
This agrees with the table in \eqref{eq:valenceTable}, when $n=0$, except that in the row $j=8$, the lower bound on $B_{k,\ell}$ in \eqref{eq:BklTable} is one less than in \eqref{eq:valenceTable}.  Using $A_{k,\ell} \geq 1$ for $j=8$, then shows that the table in \eqref{eq:valenceTable} holds also for $n=0$, and thus shows Proposition \ref{prop:nozerosonarc}.

Now we proceed to show the existence of this zero.  Lemma \ref{lemma:FkapproxOnArc} holds when $n=0$, for $\ell \geq 14$.  We have
\begin{equation}
 M_{\ell+8,\ell}(\theta) = 2\cos(4 \theta) +2\cos(4 \theta + \tfrac{\ell \theta}{2})(2i\sin(\tfrac{\theta}{2}))^{-\ell} + 2\cos(\tfrac{\ell \theta}{2})(2i\sin(\tfrac{\theta}{2}))^{-\ell - 8}.
\end{equation}
Note that
\begin{equation}
 |M_{\ell+8,\ell}(\tfrac{\pi}{2}) -2| \leq 2^{2-\frac{\ell}{2}},
\end{equation}
so we need to find a negative value of $M_{\ell+8,\ell}(\theta)$.
When $\ell \equiv 2 \pmod{6}$, we have
$ M_{\ell+8,\ell}(\pi/3) = -3$, which easily gives the desired sign change.
We need to work a little harder for the other congruence classes.

We shall choose $\phi_{\ell} = \frac{\pi}{3} + \frac{\pi}{2\ell}$ as a sample point.  We have
\begin{equation}
 2 \cos(4 \phi_{\ell}) = -1 + O(\ell^{-1}), 
\end{equation}
and by Lemma \ref{lemma:sinMonotonic},
\begin{equation}
 (2 \sin(\phi_{\ell}/2))^{-\ell} = (2 \sin(\tfrac{\pi}{6} + \tfrac{\pi}{4 \ell}))^{-\ell} \leq (2 \sin(\tfrac{\pi}{6} + \tfrac{\pi}{56}))^{-14} < 0.3.
\end{equation}
Therefore, for $\ell$ large enough, this gives another sign change.

\section{Approximation formulas for the Eisenstein series}
In this section we derive some formulas that will aid our understanding of $E_k$ along the line $x=1/2$.  Our mindset has been to obtain asymptotic approximations and to avoid estimates with explicit constants, in order to simplify the presentation.  One basic feature is that $E_k$ is close to $1$, so we need to focus on the behavior of $E_k - 1$.  Actually, it is helpful to re-scale, and instead look at 
\begin{equation}
\label{eq:GandHdefinitions}
G_k(z) := z^k(E_k(z) - 1),\qquad \text{and} \qquad H_k(z):=|z|^k (E_k(z) - 1).  
\end{equation}
Observe that $G_k(e^{i\theta}) = F_k(\theta)$.
We summarize the main results to be proved in this section with the following
\begin{mytheo}
\label{thm:EisensteinApproxTheorem}
Suppose that $z \in \mathcal{F}$.  If $y \leq k^{2/5}$, then 
\begin{equation}
\label{eq:EisensteinApproxTheoremSmally}
G_k(z) = 1 + \frac{z^k}{(z-1)^k} + \frac{z^k}{(z+1)^k} + O(\exp(-k^{1/6})).
\end{equation}
If $k^{2/5} < y \ll k^{2/3}$, then 
with the definition
\begin{equation}
\label{eq:rdef}
r = \frac{2 \pi y^2}{k},
\end{equation}
and with $\theta(z, \tau)$ denoting the Jacobi theta function as in \eqref{eq:JacobithetaDefinition}, 
we have
\begin{equation}
\label{eq:EisensteinApproxTheoremLargey}
G_k(z) =  \theta\Big(\frac{k}{2\pi y} + \frac{ix}{r}, \frac{i}{r}\Big) + O\Big(\frac{y}{k^{2/3}}\Big).
\end{equation}
\end{mytheo}
Remark.  Using the modularity relation \eqref{eq:JacobithetaModularity}, one may use the alternative formula \eqref{eq:thetaFormulaAfterModularity} in \eqref{eq:EisensteinApproxTheoremLargey}.

We begin with a bound that estimates the terms in \eqref{eq:EkczdDefinition} with $|c| \geq 2$.
\begin{mylemma}
\label{lemma:tailboundCatleast2}
Suppose $z \in \mathcal{F}$. We have
 \begin{equation}
 |z|^k  \sum_{|c| \geq 2} \sum_{|d| \geq 1} |cz+d|^{-k} \ll 3^{-k/2} \Big( 1 + \frac{y}{\sqrt{k}} \Big).
 \end{equation}
\end{mylemma}
\begin{proof}
We begin with an elementary observation.  Suppose $f(t)$ is a continuous non-negative  function integrable over $\mathbb{R}$ having a unique local maximum $t_0 \in \mathbb{R}$ and such that $f$ is increasing for $t < t_0$, and decreasing for $t > t_0$.  Then
\begin{equation}
\label{eq:elementaryBound}
 \sum_{d \in \mathbb{Z}} f(d) \leq f(t_0) + \intR f(t) dt.
\end{equation}
For a proof, suppose without loss of generality that $t_0 \in [0,1]$.  Then we have $\sum_{d \geq 1} f(d) \leq f(1) + \int_1^{\infty} f(t) dt$, by the integral test, and similarly, $\sum_{d \leq 0} f(d) \leq f(0) + \int_{-\infty}^{0} f(t) dt$.  Finally, we note that $f(0) + f(1) \leq f(t_0) + \int_0^{1} f(t) dt$, since the minimum of $f(0)$ and $f(1)$ is the minimum of $f(t)$ over the interval $[0,1]$, whence $\min(f(0), f(1)) \leq \int_0^{1} f(t) dt$, while the maximum of $f(0)$ and $f(1)$ is at most $f(t_0)$.


Using \eqref{eq:elementaryBound}
 we have
\begin{align}
  \sum_{|c| \geq 2} \sum_{|d| \geq 1} ((cx+d)^2 + c^2 y^2)^{-k/2} 
 \leq 2 \sum_{c=2}^{\infty} \Big((cy)^{-k} + \intR (t^2 + c^2 y^2)^{-k/2}  dt \Big),
\end{align}
which equals
\begin{equation}
 2 \Big( \frac{\zeta(k)-1}{y^k} + \frac{\zeta(k-1)-1}{y^{k-1}} I_k \Big), \quad I_k = \int_{-\infty}^{\infty} \frac{dt}{(t^2+1)^{k/2}}.
\end{equation}
It is easy to see $\zeta(d) = 1 + O(2^{-d})$ for $d \geq 2$.
Changing variables $t = \sqrt{u}$, and using \cite[(3.194.3)]{GR} (Mathematica can also evaluate the integral), we have
\begin{equation}
 I_k = B\Big(\frac12, \frac{k-1}{2}\Big) = \sqrt{\pi} \frac{\Gamma\Big(\frac{k-1}{2}\Big)}{\Gamma\Big(\frac{k}{2}\Big)} = \frac{4 \pi}{2^{k}} \frac{\Gamma(k)}{\Gamma(k/2)^{2}}.
\end{equation}
The last displayed equation follows on using the duplication formula.  Stirling's formula shows $I_k \ll k^{-1/2}$.  In all, we obtain
\begin{equation}
\label{eq:cbiggerthan2rawbound}
 \sum_{|c| \geq 2} \sum_{|d| \geq 1} |cz+d|^{-k} \ll (2y)^{-k} \Big(1 + \frac{y}{\sqrt{k}}\Big).
\end{equation}
Next we note, using $y \geq \frac{\sqrt{3}}{2}$, that
\begin{equation}
\label{eq:zover2ybound}
 \Big(\frac{|z|}{2y} \Big)^2 \leq \frac{\frac14 + y^2}{4y^2} \leq \frac14 + \frac{1}{16(3/4)} = \frac13.
\end{equation}
Combining \eqref{eq:zover2ybound} with \eqref{eq:cbiggerthan2rawbound} finishes the proof.
\end{proof}
\begin{mycoro}
\label{coro:EisensteinApproxwithc1}
 For $z \in \mathcal{F}$, we have
 \begin{equation}
 \label{eq:GkapproxCgeq2}
  G_k(z) = z^k \sum_{d \in \mz} (z+d)^{-k} + O\Big(3^{-k/2} \Big( 1 + \frac{y}{\sqrt{k}}\Big)\Big).
 \end{equation}
\end{mycoro}
The re-scaling by $z^k$ causes the term $d=0$ alone to give a summand of $1$ on the right hand side, which shows that this error term is indeed very small for $k$ large.

\begin{mylemma}
\label{lemma:dsumtailestimategeneralD}
Suppose $D \geq 1$ and $z \in \mathcal{F}$.  Then
\begin{equation}
\sum_{|d| \geq D} \frac{|z|^k}{|z+d|^k} \ll y \Big(\frac{\frac14 + y^2}{(D-\frac12)^2+y^2}\Big)^{k/2}.
\end{equation}
\end{mylemma}
Remark.  It can be checked that if $D=2$, this error term is $o(1)$ provided $y = o(k^{1/2} (\log k)^{-1/2})$.  
\begin{proof}
To begin, we obviously have
\begin{equation}
\sum_{|d| \geq D} |z+d|^{-k} =  \sum_{|d| \geq D} ((x+d)^2 + y^2)^{-k/2}.
\end{equation}
For the terms with $d \geq D$, the sum is maximized, in terms of $x$, at $x=-1/2$, in which case 
\begin{equation}
\sum_{d\geq D}^{\infty} ((-1/2+d)^2 + y^2)^{-\frac{k}{2}} \leq ((D-1/2)^2 + y^2)^{-\frac{k}{2}} + \int_{D}^{\infty} ((t-\tfrac12)^2 + y^2)^{-\frac{k}{2}} dt,
\end{equation}
by the integral comparison test.  We estimate the integral as follows:
\begin{equation}
 \int_{D}^{\infty} ((t-\tfrac12)^2 + y^2)^{-\frac{k}{2}} dt \leq ((D-\tfrac12)^2 + y^2)^{1-\frac{k}{2}} \int_{D-\frac12}^{\infty} \frac{dt}{t^2 + y^2} \ll y ((D-\tfrac12)^2+y^2)^{-\frac{k}{2}}.
\end{equation}
Accounting for $d \leq -D$ at worst doubles the bound.  Using $|z|^2 \leq \frac14 + y^2$ and combining the bounds completes the proof.
\end{proof}

\begin{mycoro}
\label{coro:EisensteinApproxwithc2}
Suppose $z \in \mathcal{F}$ and $y \leq k^{2/5}$.  Then
\begin{equation}
\sum_{d \in \mz} \frac{z^k}{(z+d)^k} = 1 + \frac{z^k}{(z-1)^k} + \frac{z^k}{(z+1)^k} + O(y\exp(- k^{1/5})).
\end{equation}
\end{mycoro}
Corollaries \ref{coro:EisensteinApproxwithc1} and \ref{coro:EisensteinApproxwithc2}, when combined, show \eqref{eq:EisensteinApproxTheoremSmally}.
\begin{proof}
We take $D=2$ in Lemma \ref{lemma:dsumtailestimategeneralD}.  The expression in parentheses on the right hand side is monotonically increasing in $y$, so suppose $y \leq Y = k^{2/5}$.  Then  we have
\begin{equation}
\label{eq:largeYapproximation}
\Big(\frac{\frac14 + Y^2}{\frac94+Y^2}\Big)^{k/2} = \exp\Big(\frac{k}{2} \log\Big(\frac{1 + \frac{1}{4Y^2}}{1 + \frac{9}{4Y^2}} \Big) \Big) = \exp(\tfrac{k}{2} (-2Y^{-2} + O(Y^{-4}))).
\end{equation}
By assumption, $Y^{-4} k \ll 1$, so we immediately derive the result.
\end{proof}

%
%

\begin{mylemma}
Suppose $z \in \mathcal{F}$ and $k^{2/5} < y \leq k$.  
Suppose that $D \in \mathbb{N}$ satisfies
\begin{equation}
\label{eq:Dupperboundcondition}
\frac{k D^3}{y^3} \ll 1, \qquad \text{and} \qquad \frac{k D^2}{y^2} \gg 1.
\end{equation}
Then
\begin{equation}
\label{eq:thetafunctionapproximationPrototype}
\sum_{|d| \leq D} \frac{z^k}{(z+d)^k} = \sum_{d \in \mz} \exp\Big(\frac{ikd}{y} - \frac{xdk}{y^2} - \frac{k d^2}{2y^2}\Big) + O\Big(\frac{kD^4}{|z|^3}\Big) + O\Big(\Big(1 + \frac{y}{\sqrt{k}}\Big) \exp\Big(-\frac{kD^2}{2y^2}\Big) \Big).
\end{equation}
\end{mylemma}
Remark.  The conditions in \eqref{eq:Dupperboundcondition} are equivalent to
\begin{equation}
\label{eq:Dsize}
\frac{y}{k^{1/2}} \ll D \ll \frac{y}{k^{1/3}},
\end{equation}
and since we assume $y \gg k^{2/5}$, there is a positive integer satisfying \eqref{eq:Dsize} for large $k$.

\begin{proof}
We have
\begin{equation}
\sum_{|d| \leq D} \Big(1 + \frac{d}{z}\Big)^{-k} = \sum_{|d| \leq D} \exp\Big(-k \log\Big(1 + \frac{d}{z}\Big)\Big).
\end{equation}
Next we use a Taylor approximation for $\log$ to give
\begin{equation}
\label{eq:logTaylorApprox}
-k \log\Big(1 + \frac{d}{z}\Big) = -k\frac{d}{z} + \frac{k}{2} \frac{d^2}{z^2} + O\Big(\frac{kd^3}{|z|^3}\Big).
\end{equation}
By \eqref{eq:Dsize}, this error term is $O(1)$, so taking exponentials gives
\begin{equation}
\label{eq:SomeRandomExponentialFormula}
\sum_{|d| \leq D} \Big(1 + \frac{d}{z}\Big)^{-k} = \sum_{|d| \leq D} \exp\Big(-k\frac{d}{z} + \frac{kd^2}{2z^2}\Big) + O\Big(\frac{kD^4}{|z|^3}\Big).
\end{equation}
This error term appears in \eqref{eq:thetafunctionapproximationPrototype}.
Next we write $z = iy(1 + \frac{x}{iy})$ and take Taylor approximations again, giving that the argument of the exponential in \eqref{eq:SomeRandomExponentialFormula} is
\begin{equation}
-k\frac{d}{ iy(1 + \frac{x}{iy})} + \frac{kd^2}{2 (iy)^2(1 + \frac{x}{iy})^2} = \frac{ikd}{y} - \frac{xdk}{y^2} - \frac{k d^2}{2y^2} + O\Big(\frac{k d^2}{y^3}\Big).
\end{equation}
This error term, when applied to \eqref{eq:SomeRandomExponentialFormula}, is absorbed by that in \eqref{eq:logTaylorApprox}, so we derive
\begin{equation}
\sum_{|d| \leq D} \Big(1 + \frac{d}{z}\Big)^{-k} = \sum_{|d| \leq D} \exp\Big(\frac{ikd}{y} - \frac{xdk}{y^2} - \frac{k d^2}{2y^2}\Big) + O\Big(\frac{kD^4}{|z|^3}\Big).
\end{equation}
Under the assumption $D \gg \frac{y}{\sqrt{k}}$, we shall extend the sum back to all $d \in \mz$.  To bound the error in this extension, we have, using $|x| \leq 1/2$, and then the integral comparison test, that
\begin{multline}
\sum_{|d| > D} \exp\Big( - \frac{xdk}{y^2} - \frac{k d^2}{2y^2}\Big) \leq 2 \sum_{d > D} \exp\Big( -\frac{k}{2y^2} (d^2-d) \Big) 
\\
\leq 2\exp\Big(-\frac{k}{2y^2}((D+1)^2-(D+1)\Big) + 2\int_{D+1}^{\infty} \exp\Big(-\frac{k}{2y^2}(t^2-t)\Big) dt.
\end{multline}
The first summand is $\leq 2\exp(-\frac{k D^2}{2y^2})$.  By changing variables, we have
\begin{equation}
\int_{D+1}^{\infty} \exp\Big(-\frac{k}{2y^2}(t^2-t)\Big) dt \ll \frac{y}{\sqrt{k}} \int_{\frac{D \sqrt{k}}{y \sqrt{2}}}^{\infty} \exp(-v^2) dv.
\end{equation}
Then using the bound $\erfc(x) \ll e^{-x^2}$ for $x \geq 0$, we obtain
\begin{equation}
\sum_{|d| > D} \exp\Big( - \frac{xdk}{y^2} - \frac{k d^2}{2y^2}\Big) \ll (1 + \frac{y}{\sqrt{k}}) \exp\Big(-\frac{kD^2}{2y^2}\Big).
\end{equation}
Combining the various error terms completes the proof.
\end{proof}

Using the notation of the Jacobi theta function, and recalling the definition of $r$ from \eqref{eq:rdef}, we have
\begin{equation}
\sum_{d \in \mz} \exp\Big(\frac{ikd}{y} - \frac{xdk}{y^2} - \frac{k d^2}{2y^2}\Big) = \theta\Big(\frac{k}{2\pi y} + \frac{ix}{r}, \frac{i}{r}\Big).
\end{equation}

Combining these results, we obtain that if \eqref{eq:Dupperboundcondition} holds, and $k^{2/5} < y \leq k$, then
\begin{multline}
\label{eq:thetafunctionapproximationPrototype2}
\sum_{d \in \mz} \frac{z^k}{(z+d)^k} = \theta\Big(\frac{k}{2\pi y} + \frac{ix}{r}, \frac{i}{r}\Big)  + O\Big(\frac{kD^4}{|z|^3}\Big) + O\Big((1 + \frac{y}{\sqrt{k}}) \exp\Big(-\frac{kD^2}{2y^2}\Big) \Big)
\\ 
+ O\Big(y \Big(\frac{\frac14 + y^2}{(D-\frac12)^2+y^2}\Big)^{k/2} \Big).
\end{multline}
We shall simplify the error terms in \eqref{eq:thetafunctionapproximationPrototype2}.
One may check by a similar argument to \eqref{eq:largeYapproximation} that
\begin{equation}
\Big(\frac{\frac14 + y^2}{(D-\frac12)^2+y^2}\Big)^{k/2} = \exp\Big(-\frac{k}{2} \Big[\frac{(D-\frac12)^2-\frac14}{y^2} + O\Big(\frac{D^4}{y^4}\Big)\Big] \Big).
\end{equation}
Since $D^4 k y^{-4} = o(1)$, the error term can be discarded to give in all that
\begin{equation}
\sum_{d \in \mz} \frac{z^k}{(z+d)^k} = \theta\Big(\frac{k}{2\pi y} + \frac{ix}{r}, \frac{i}{r}\Big)  + O\Big(\frac{kD^4}{|z|^3}\Big) + O\Big(y \exp\Big(-\frac{k(D^2-D)}{2y^2}\Big) \Big).
\end{equation}
Since $D$ is a parameter at our disposal subject to \eqref{eq:Dupperboundcondition}, we may choose $D = y/k^{5/12} + O(1)$ to give
\begin{mycoro}
\label{coro:dsumapproximationlargey}
Suppose that $z \in \mathcal{F}$ and $k^{2/5} < y \leq k$.  Then
\begin{equation}
\sum_{d \in \mz} \frac{z^k}{(z+d)^k} = \theta\Big(\frac{k}{2\pi y} + \frac{ix}{r}, \frac{i}{r}\Big) + O\Big(\frac{y}{k^{2/3}}\Big).
\end{equation}
\end{mycoro}
Combining Corollary \ref{coro:dsumapproximationlargey} with Corollary \ref{coro:EisensteinApproxwithc1} gives \eqref{eq:EisensteinApproxTheoremLargey}.  This completes the proof of Theorem \ref{thm:EisensteinApproxTheorem}.

Next we derive an alternative form of the main term.  Using \eqref{eq:JacobithetaModularity} with $w = -x+ \frac{kir}{2 \pi y}  = -x+iy$, and $\tau = ir$,  we obtain
\begin{equation}
 \theta\Big(\frac{k}{2 \pi y} + \frac{ix}{r}, \frac{i}{r}\Big) = r^{1/2} \exp\Big(\frac{\pi}{r}(-x+iy)^2\Big) \theta(-x + iy, ir).
\end{equation}
Writing out the definition of the theta function, we deduce
\begin{equation}
 \theta\Big(\frac{k}{2 \pi y} + \frac{ix}{r}, \frac{i}{r}\Big) = r^{1/2} \exp\Big(\frac{\pi}{r}(-x+iy)^2\Big) \sum_{n \in \mz} \exp(-\pi n^2 r -2\pi ny - 2\pi i nx).
\end{equation}
Completing the square, and replacing $n$ by $-n$, it simplifies as
\begin{equation}
\label{eq:thetaFormulaAfterModularity}
 \theta\Big(\frac{k}{2 \pi y} + \frac{ix}{r}, \frac{i}{r}\Big) = r^{1/2} 
 \exp\Big(\frac{-ikx}{y} + \frac{\pi x^2}{r} \Big)
  \sum_{n \in \mz} \exp\Big(-\pi r \Big(n-\frac{k}{2 \pi y}\Big)^2\Big) \exp( 2 \pi i nx).
\end{equation}

Finally, we focus on the largest values of $y$, which are naturally treated with the Fourier expansion.  Rather than directly quoting the Fourier expansion of $E_k$ itself, in light of Lemma \ref{lemma:tailboundCatleast2}, we can proceed as follows.  First we quote the well-known formula (e.g. see \cite[p.16]{Zagier}):
\begin{equation}
\label{eq:Zagierformula}
 \sum_{d \in \mz} (z+d)^{-k} = \frac{(-2\pi i)^k}{\Gamma(k)} \sum_{n=1}^{\infty} n^{k-1} e(nz).
\end{equation}
\begin{mylemma}
\label{lemma:EkApproxFourier}
Suppose $C \geq 1$, and $y \geq 1$. Then
\begin{equation}
\label{eq:EkApproxFourier}
\frac{(-2\pi i y)^k}{\Gamma(k)}  
\sum_{n=1}^{\infty} n^{k-1} e(nz) = \frac{(-2\pi i y)^k}{\Gamma(k)}  
\sum_{|n-\frac{k}{2 \pi y} | \leq C \frac{k^{1/2}}{2\pi y}} n^{k-1} e(nz) + O(e^{-C/4}).
\end{equation}
The implied constant is absolute.
\end{mylemma}
Remarks.  To gauge the size of the main term on the right hand side of \eqref{eq:EkApproxFourier}, we 
suppose for simplicity that $\frac{k}{2\pi y}$ is a positive integer.  Then this term alone contributes to the right hand side a quantity of the size
\begin{equation}
\label{eq:EkSingleTermSizeFourier}
y^k k^{1/2} \Big(\frac{2 \pi e}{k} \Big)^k \Big(\frac{k}{2 \pi y}\Big)^{k-1} e^{-k} \asymp \frac{y}{\sqrt{k}}.
\end{equation}
Therefore, the error term is good if $y \gg \sqrt{k}$ with a large enough implied constant (which then means the main term in \eqref{eq:EkApproxFourier} does indeed consist of a single term).  In the complementary range $y \ll \sqrt{k}$, with a small enough implied constant, we have that all the terms in the sum are of size \eqref{eq:EkSingleTermSizeFourier}, but now there are $\approx  k^{1/2}/y$ such terms, showing that the right hand side is roughly of size $\asymp 1$.

\begin{proof}
The function $f(t) = t^{k-1} \exp(-2 \pi t y)$ is initially increasing and eventually decreasing, with a unique local maximum occuring at $t_0 = \frac{k-1}{2 \pi y}$.  If $N > t_0$, then by the integral comparison test,
\begin{equation}
\sum_{n=N+1}^{\infty} n^{k-1} \exp(- 2 \pi n y) \leq \int_N^{\infty} t^{k-1} \exp(- 2\pi ty) dt.
\end{equation}
Using the notation
\begin{equation}
 Q(a,x) = \frac{1}{\Gamma(a)} \int_x^{\infty} t^{a} e^{-t} \frac{dt}{t}
\end{equation}
for the normalized incomplete gamma function, we have that
\begin{equation}
\Big| \frac{(2\pi  y)^k}{\Gamma(k)}  \sum_{n \geq N+1} n^{k-1} e(nz) \Big| \ll  Q(k, 2 \pi N y).
\end{equation}
By a similar argument, if $M < t_0$, then
\begin{equation}
\Big| \frac{(2\pi  y)^k}{\Gamma(k)}   \sum_{n \leq M-1} n^{k-1} e(nz) \Big| \ll  P(k, 2 \pi M y),
\end{equation}
where $P(a,x)$ is the complementary incomplete gamma function defined by
\begin{equation}
 P(a,x) = \frac{1}{\Gamma(a)} \int_0^{x} t^{a} e^{-t} \frac{dt}{t} = 1 - Q(a,x).
\end{equation}
One may derive from work of Temme \cite{Temme} that if $x > a$, then 
\begin{equation}
\label{eq:Qbound}
Q(a,x) \ll e^{-\frac{(x-a)^2}{4a}} + e^{-\frac{|x-a|}{4}}.
\end{equation}
For this, \cite[(1.4)]{Temme} gives 
\begin{equation}
Q(a,x) \ll \erfc(\eta (a/2)^{1/2}) + \exp(-\tfrac12 \eta^2 a) \ll \exp(-\tfrac12 \eta^2 a),
\end{equation}
where $\frac12 \eta^2 = \mu - \log(1 + \mu)$, and $\mu = \frac{x-a}{a}$.  It is easy to check that 
\begin{equation}
t - \log(1+ t) \geq 
\begin{cases}
t^2/4, \qquad  0 \leq t \leq 1, \\
t/4, \qquad  1 \leq t < \infty,
\end{cases}
\end{equation}
whence we derive
\begin{equation}
 Q(a,x) \ll \exp\Big(-a \max\Big( \frac{(x-a)^2}{4a^2}, \frac{|x-a|}{4a}\Big)\Big),
\end{equation}
which is equivalent to \eqref{eq:Qbound}.
The same bound holds for $P(a,x)$ provided $a > x$.  In our application, we have $a=k$, and $x = 2 \pi Ny$.  If $2\pi Ny - k \geq C k^{1/2}$, then Temme's bound shows
\begin{equation}
Q(k, 2\pi N y) \ll \exp(-C/4),
\end{equation}
and likewise for $P(k, 2\pi My)$ provided $k- 2\pi My  \geq C k^{1/2}$.
\end{proof}

Lemma \ref{lemma:EkApproxFourier} leads to a very close approximation if say $y \geq k^{3/5}$ (anything slightly larger than $k^{1/2}$ would do):
\begin{mycoro}
\label{coro:EisensteinApproxFourierLargey}
 Suppose $z \in \mathcal{F}$ and $y \geq k^{3/5}$.  Then we have
 \begin{equation}
 G_k(z) = \frac{(-2\pi i z)^k}{\Gamma(k)}  
\sum_{|n-\frac{k}{2 \pi y} | \leq (\log k)^2 \frac{k^{1/2}}{2\pi y}} n^{k-1} e(nz) + O(\exp(-c(\log^2 k))),
 \end{equation}
 where $c > 0$ is some constant independent of $y$ and $k$.
\end{mycoro}
For this, all we need to notice is that for $y \gg \sqrt{k}$ and $|x| \leq 1/2$, that $|z/y|^k \asymp 1$.

Next we specialize
Theorem \ref{thm:EisensteinApproxTheorem} to the case where $x=1/2$ which is required for understanding zeros on the sides of the fundamental domain.  The formulas simplify.  In polar coordinates, $1/2 + iy = R e^{i \theta}$ where $R = (1/4 + y^2)^{1/2}$.  With $\phi$ defined by $\theta + \phi = \frac{\pi}{2}$ (geometrically, $\phi$ is the angle measured from the $y$-axis), then $\tan(\phi) = \frac{1}{2y}$ (and of course $\tan(\theta) = 2y$).
\begin{mycoro}
\label{coro:Hkapproximations}
 Suppose that $z \in \mathcal{F}$.  If $y \leq k^{2/5}$, then 
\begin{align}
\label{eq:EisensteinSideApproxsmally}
H_k(1/2+iy) &= 2 (-1)^{k/2} \cos(k \phi) + O(\exp(-k^{1/6}))
\\
\label{eq:EisensteinSideApproxsmallySecondForm}
&=2 \cos(k \theta ) + O(\exp(-k^{1/6})).
\end{align}
If $k^{2/5} < y \ll k^{1/2}$, then 
\begin{equation}
\label{eq:EisensteinSideApproxsmallishy}
\Big| H_k(1/2+iy) - 2 (-1)^{k/2} \cos(k \phi) \Big| \leq 
 \Phi_0(r) +  O(k^{-1/6}),
\end{equation}
where recall $r = \frac{2 \pi y^2}{k}$, and 
\begin{equation}
 \Phi_0(r) = \sum_{\substack{n \in \mz \\ n \neq 0,-1}} \exp\Big(-\frac{\pi}{r} (n^2 + n)\Big).
\end{equation}
If $k^{1/2} \ll y \ll k^{3/5}$, and $y=y_N$ satisfies 
\begin{equation}
\label{eq:yNapproxformula}
 y_N = \frac{k}{2 \pi N} (1 + O(k^{-1})),
\end{equation}
for $N \in \mathbb{N}$, then
\begin{equation}
\label{eq:EisensteinSideApproxlargey}
  \Big| H_k(1/2+iy_N)  - (-1)^{N+\frac{k}{2}}  r^{1/2} 
 \exp\Big(\frac{\pi }{4r}\Big) \Big|  \leq  r^{1/2} 
 \exp\Big(\frac{\pi }{4r}\Big) \Phi_1(r) + O(k^{-1/15}),
 \end{equation}
where
\begin{equation}
 \Phi_1(r) = 
 \sum_{n \neq 0} \exp(-\pi r n^2).
\end{equation}
Finally, if $k^{3/5} \ll y \ll k$, and $y=y_N$ satisfies \eqref{eq:yNapproxformula}, 
then 
\begin{equation}
\label{eq:EisensteinSideApproxlargesty}
H_k(1/2+iy_N) = (-1)^{N+\frac{k}{2}} r^{1/2} (1 + O(k^{-1/5})).
 \end{equation}
\end{mycoro}
In Figure \ref{fig:Phi0} we include plots of $\Phi_0$ and $\Phi_1$. 
\begin{figure}[h]
\begin{center}
\scalebox{0.7}[0.7]{\includegraphics{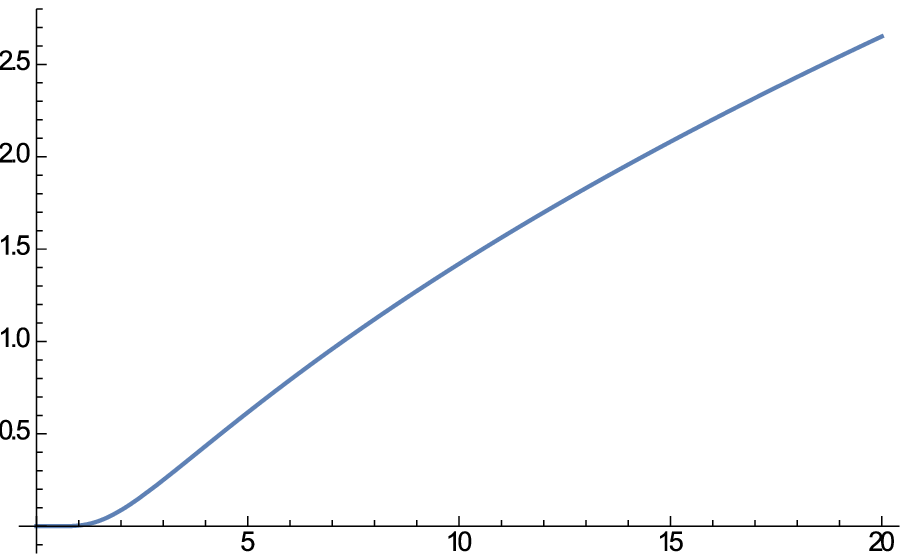}}
\scalebox{0.7}[0.7]{\includegraphics{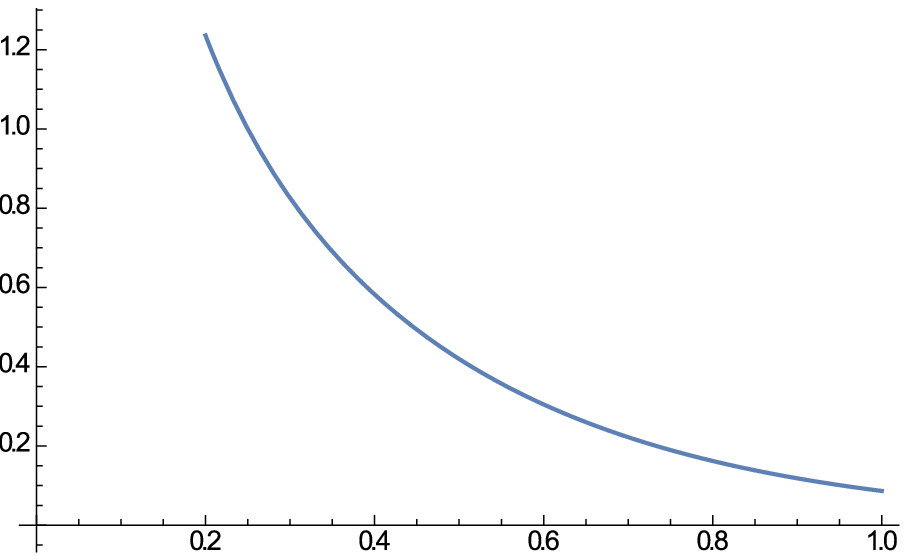}}
\caption{Plot of $\Phi_0(r)$ (left) and $\Phi_1(r)$ (right)}
\label{fig:Phi0}
\end{center}
\end{figure}
\begin{proof}
Note 
\begin{equation}
 \Big(\frac{|1/2+iy|}{1/2+iy}\Big)^k = e^{-i \theta k} = (-1)^{k/2}e^{ik \phi} = (-1)^{k/2} e^{ik \arctan(\frac{1}{2y})},
\end{equation}
and so
\begin{equation}
 \Big(\frac{|1/2+iy|}{-1/2+iy}\Big)^k =e^{i \theta k} = (-1)^{k/2} e^{-ik \phi} =  (-1)^{k/2} e^{-ik \arctan(\frac{1}{2y})}.
\end{equation}
In the range $y \leq k^{2/5}$, then using Corollaries \ref{coro:EisensteinApproxwithc2} and \ref{coro:EisensteinApproxwithc2}, 
we derive
\begin{equation}
 G_k(1/2 + iy) = 1 + \frac{(1/2+iy)^k}{(-1/2+iy)^k} + O(y\exp(-k^{1/5})).
\end{equation}
Since $H_k(z) = (|z|/z)^k G_k(z)$, we then derive
\begin{equation}
 H_k(1/2+iy) = 2 (-1)^{k/2} \cos(k \arctan \tfrac{1}{2y}) + O(y\exp(-k^{1/5})),
\end{equation}
which gives \eqref{eq:EisensteinSideApproxsmally}.

Now suppose $k^{2/5} < y \ll k^{2/3}$.  Then \eqref{eq:EisensteinApproxTheoremLargey} gives
\begin{equation}
\label{eq:EisensteinApproxWithTheta}
 H_k(1/2+iy) = (-1)^{k/2} e^{ik \arctan(\frac{1}{2y})} \theta\Big(\frac{k}{2 \pi y} + \frac{i}{2r}, \frac{i}{r}\Big) + O(k^{-2/3} y),
\end{equation}
which upon writing the definition of the theta function, gives
\begin{equation}
 H_k(1/2+iy) =  (-1)^{k/2} e^{ik \arctan(\frac{1}{2y})} \sum_{n \in \mz} \exp\Big(-\frac{\pi}{r} (n^2 + n)\Big) e^{ink/y} + O(k^{-2/3} y).
\end{equation}
Note that $n^2 +n \geq 0$ for all $n \in \mz$, and that for $n \neq 0, -1$, then $n^2 + n \geq 2$.  Therefore, if $r \gg 1$ we get a good approximation by selecting the terms $n=0,-1$, and trivially bounding the rest.  
This gives
\begin{equation}
|H_k(1/2+iy) - (-1)^{k/2} e^{ik \arctan(\frac{1}{2y})} (1+e^{-ik/y})| \leq  \Phi_0(r)  +  O(k^{-2/3} y).
\end{equation}
Since $y \gg k^{2/5}$, and $\arctan(\frac{1}{2y}) = \frac{1}{2y} + O(y^{-3})$, we have
\begin{equation}
 e^{ik \arctan(\frac{1}{2y})} (1+e^{-ik/y}) = 2  \cos(k \arctan \tfrac{1}{2y}) + O(k^{-1/5}).
\end{equation}
This gives \eqref{eq:EisensteinSideApproxsmallishy}, using $k^{-2/3} y \ll k^{-1/6}$.

 In the region $k^{1/2} \ll y \ll k^{3/5}$ (which implies
 $r \gg 1$), then we do not get a close approximation by dropping all terms but $n=0,-1$.  In this case, we instead use \eqref{eq:thetaFormulaAfterModularity} in \eqref{eq:EisensteinApproxWithTheta}.  
Simplifying, we get
 \begin{equation}
 H_k(1/2+iy)
 =  (-1)^{k/2} r^{1/2} 
 \exp\Big(\frac{\pi }{4r}\Big) 
  \sum_{n \in \mz} (-1)^n \exp\Big(-\pi r \Big(n-\frac{k}{2 \pi y}\Big)^2\Big)  + O(k^{-1/15}).
 \end{equation}
Under the additional assumption \eqref{eq:yNapproxformula}, we may change variables $n \rightarrow n+N$, giving
\begin{equation}
\label{eq:HkyNapproxMiddleOfArgument}
 H_k(1/2+iy_N)
 = (-1)^{N+\frac{k}{2}}  r^{1/2} 
 \exp\Big(\frac{\pi }{4r}\Big) 
  \sum_{n \in \mz} (-1)^n  \exp(-\pi r (n + O(k^{-1}N))^2)  + O(k^{-1/15}).
 \end{equation}
 Note that $k^{2/5} \ll N \ll k^{1/2}$, so in particular $k^{-1} N \ll k^{-1/2}$.  
 
We now show that in \eqref{eq:HkyNapproxMiddleOfArgument} we may remove the error term in the exponential without making a new error term.  
 For $n=0$, we have
 \begin{equation}
  \exp(O(r k^{-2} N^2)) = 1 + O(k^{-1}), 
 \end{equation}
 which leads to an acceptable error term since $k^{-1} r^{1/2} \ll y_N k^{-3/2} \ll  k^{-9/10}$.
 Meanwhile,
for $n \neq 0$ we may assume $n \ll r^{-1/2} (\log k)^2$, whence $k^{-1} r |n| N \ll k^{-1/2} (\log k)^2$ .  Then we have
\begin{multline}
 \sum_{0 < |n| \ll \frac{(\log k)^2}{r^{1/2}} } (-1)^n  \exp(-\pi r (n + O(k^{-1}N))^2) 
 \\
 = \sum_{0 < |n| \ll \frac{(\log k)^2}{r^{1/2}}} (-1)^n  \exp(-\pi r n^2) ( 1 + O(k^{-1} r |n| N))
 = \sum_{n \neq 0}  (-1)^n  \exp(-\pi r n^2) + O(k^{-1/2}).
\end{multline}
That is, we have shown
\begin{equation}
\label{eq:HkyNapproxMiddleOfArgument2}
 H_k(1/2+iy_N)
 = (-1)^{N+\frac{k}{2}}  r^{1/2} 
 \exp\Big(\frac{\pi }{4r}\Big) 
  \sum_{n \in \mz} (-1)^n  \exp(-\pi r n^2)  + O(k^{-1/15}).
 \end{equation}

Now the term $n=0$ gives a good approximation, so by bounding the tail trivially we obtain \eqref{eq:EisensteinSideApproxlargey}.

Finally we turn to \eqref{eq:EisensteinSideApproxlargesty}.  By Corollary \ref{coro:EisensteinApproxFourierLargey}, we have that $n=N$ is the only term appearing in the sum, giving
\begin{equation}
\label{eq:SomeRandomFormula}
 H_k(1/2+iy_N) = \frac{(-2\pi i |1/2+iy_N|)^k}{\Gamma(k)}  
 N^{k-1} (-1)^N \exp(- 2 \pi N y_N) + O(\exp(-c(\log^2 k))).
\end{equation}
To simplify this, we have
\begin{equation}
 (-i |1/2 + iy_N|)^k = (-1)^{k/2} y_N^k \exp( O(y_N^{-2} k))
 = (-1)^{k/2} y_N^k (1 + O(k^{-1/5}),
\end{equation}
since $y \gg k^{3/5}$.  Define $\delta$ by $N = \frac{k}{2\pi y_N}(1 + \delta)$, whence \eqref{eq:yNapproxformula} gives $\delta \ll k^{-1}$.  The  main term on the right hand side of \eqref{eq:SomeRandomFormula} becomes
\begin{equation}
\label{eq:coffeecup}
(-1)^{N+k/2} \frac{(2\pi)^k}{\Gamma(k)} N^{k-1} y_N^k \exp(-2 \pi N y_N)
= \frac{(-1)^{N+\frac{k}{2}}}{N} \frac{(k/e)^k}{\Gamma(k) }  (1 + \delta)^k \exp(-k\delta).
\end{equation}
Note that $(1+\delta)^k \exp(-k \delta) = \exp(O(k \delta^2)) = 1 + O(k^{-1})$.
Using Stirling's formula, \eqref{eq:coffeecup} 
simplifies as
\begin{equation}
 \frac{(-1)^{N+\frac{k}{2}}}{N} \frac{(k/e)^k}{\Gamma(k)} (1 + O(k^{-1/5})) = (-1)^{N+\frac{k}{2}} r^{1/2} (1 + O(k^{-1/5})),
\end{equation}
as desired.
\end{proof}

We shall also find it convenient to record the following ``trivial" bound.
\begin{mycoro}
\label{coro:Hktrivialbound}
Suppose $z = 1/2 + iy \in \mathcal{F}$.  
Then $H_k(z) \ll (1 + \frac{y}{\sqrt{k}})$.
\end{mycoro}
\begin{proof}
The bounds in Corollary \ref{coro:Hkapproximations} are easily seen to be $O(1)$ for $y \ll k^{1/2}$ (with any fixed implied constant), so now assume $y \gg k^{1/2}$.

It suffices to show the desired bound for $G_k(z)$, since $|G_k(z)| = |H_k(z)|$.  The error term in Corollary \ref{coro:EisensteinApproxwithc1} is satisfactory, so we need to prove the bound for $z^k \sum_{d \in \mathbb{Z}} (z+d)^{-k}$.  Next we combine \eqref{eq:Zagierformula} with \eqref{eq:EkApproxFourier} with $C > 0$ large (but fixed), so that the right hand side of \eqref{eq:EkApproxFourier} has a single summand, with $n = \frac{k}{2\pi y}(1 +\delta)$, and $\delta \ll k^{-1/2}$.  We have $|z/y|^k \ll 1$, by a Taylor approximation, and similarly to the proof of \eqref{eq:EisensteinSideApproxlargesty}, we have
\begin{equation}
\frac{(2\pi y)^k}{\Gamma(k)} n^{k-1} \exp(-2 \pi ny) = \frac{2 \pi y}{k} \frac{(k/e)^k}{\Gamma(k)} (1+\delta)^{k-1} \exp(-k\delta) \ll \frac{y}{\sqrt{k}}.
\qedhere
\end{equation}
\end{proof}

We shall occasionally find it useful to study derivatives, especially at the corner point $e^{\pi i/3}$.  Our first claim is that \eqref{eq:EisensteinApproxTheoremSmally} holds uniformly in a sufficiently small neighborhood of $z = e^{\pi i/3}$.  Since this point lies on the boundary of $\mathcal{F}$, the proofs need to be adjusted slightly.  Since our previous analysis was designed to hold for large values of $y$, we found it easiest to revisit the previous proofs.  We have
\begin{mylemma}
 Suppose $2/5 \leq x \leq 3/5$ and $2^{-1/2} \leq y \leq 1$.  Then
 \begin{equation}
  G_k(z) = 1 + \frac{z^k}{(z-1)^k} + \frac{z^k}{(z+1)^k} + O((\tfrac{37}{17})^{-k/2}).
 \end{equation}
\end{mylemma}
\begin{proof}
 The bound \eqref{eq:cbiggerthan2rawbound} is valid for $z \in \mathbb{H}$, and shows
 \begin{equation}
  |z|^k \sum_{|c| \geq 2} \sum_{|d| \geq 1} |cz+d|^{-k} \ll \Big(\frac{|z|^2}{4y^2}\Big)^{k/2} \leq \Big(\frac{\frac{9}{25} + y^2}{4y^2}\Big)^{k/2} = (\tfrac{100}{43})^{-k/2},
 \end{equation}
under the additional restrictions in place here.  Note $\frac{100}{43} > \frac{37}{17}$.
For the terms with $|c| =1$, we use a method
similar to the proof of Lemma \ref{lemma:dsumtailestimategeneralD}, giving
\begin{equation}
|z|^k \sum_{|d| \geq 2} ((x+d)^2 + y^2)^{-k/2} \ll \Big(\frac{|z|^2}{(7/5)^2 + y^2}\Big)^{k/2} \leq \Big(\frac{(3/5)^2 + y^2}{(7/5)^2 + y^2}\Big)^{k/2} \leq (\tfrac{37}{17})^{-k/2}. \qedhere
\end{equation} 
\end{proof}

Using Cauchy's integral formula, we derive
\begin{mycoro}
Suppose $2/5 \leq x \leq 3/5$ and $2^{-1/2} \leq y \leq 1$.  Then for $j=1,2, \dots$, we have
 \begin{equation}
  \frac{d^j}{dz^j} G_k(z) = \frac{d^j}{dz^j} \Big(1 + \frac{z^k}{(z-1)^k} + \frac{z^k}{(z+1)^k} \Big) + O_j((\tfrac{37}{17})^{-k/2}).
 \end{equation}
\end{mycoro}
In our upcoming application, we shall want to differentiate $H_k(z)$ along a curve.  Define
 \begin{equation}
  r_k(z) = G_k(z)- \Big(1 + \frac{z^k}{(z-1)^k} + \frac{z^k}{(z+1)^k} \Big).
 \end{equation}
\begin{mylemma}
\label{lemma:HkDerivative}
 Suppose that $z(t)$ is a smooth curve, defined for $t \in (-\varepsilon, \varepsilon)$ for some small $\varepsilon > 0$, and with $z(t)$ staying in the region $2/5 \leq x \leq 3/5$ and $2^{-1/2} \leq y \leq 1$.  Then for $j=1,2$, we have
 \begin{equation}
  \frac{d^j}{dt^j} H_k(z(t)) \vert_{t=0} = \frac{d^j}{dt^j} \Big[ \Big(\frac{|z(t)|}{z(t)}\Big)^k \Big(1 + \frac{z(t)^k}{(z(t)-1)^k} + \frac{z(t)^k}{(z(t)+1)^k} \Big) \Big] \vert_{t=0} + O( 2^{-k/2}),
 \end{equation}
 where the implied constant depends on the curve $z(t)$.
\end{mylemma}
Remark.  In practice, this lemma means that we may safely approximate $H_k$ by its main term, and differentiate the main term.

\begin{proof}
Since $H_k(z) = (|z|/z)^k G_k(z)$, we have
\begin{equation}
\frac{d}{dt} H_k(z(t)) - \frac{d}{dt} \Big[ \Big(\frac{|z(t)|}{z(t)}\Big)^k \Big(1 + \frac{z(t)^k}{(z(t)-1)^k} + \frac{z(t)^k}{(z(t)+1)^k} \Big) \Big]
=
 \frac{d}{dt}  \Big(\frac{|z(t)|}{z(t)}\Big)^k r_k(z(t)),
\end{equation}
and so it suffices to bound the right hand side.  We have
\begin{equation}
 \frac{d}{dt}  \Big(\frac{|z(t)|}{z(t)}\Big)^k r_k(z(t)) = r_k(z(t)) \frac{d}{dt}  \Big(\frac{|z(t)|}{z(t)}\Big)^k + \Big(\frac{|z(t)|}{z(t)}\Big)^k r_k'(z(t)) z'(t).
\end{equation}
It is easy to see that
\begin{equation}
 \frac{d}{dt}  \Big(\frac{|z(t)|}{z(t)}\Big)^k = k \Big(\frac{|z(t)|}{z(t)}\Big)^{k-1} \frac{d}{dt} \frac{|z(t)|}{z(t)} \ll k.
\end{equation}
This proves the desired result for $j=1$.  The case $j=2$ is similar.
\end{proof}

\begin{mycoro}
\label{coro:HderivativeEstimate}
Let $1/2 + iy = R e^{i \theta}$.  Then
\begin{equation}
\frac{d}{d\theta} H_k(1/2 + iy)\vert_{\theta = \pi/3} = \frac{d}{d \theta} 2 \cos(k \theta) \vert_{\theta = \pi/3} + O( 2^{-k/2}).
\end{equation}
\end{mycoro}

\section{Zeros on the side}
\label{section:Bkl}
To understand the zeros on the line $x=1/2$, we will generally evaluate $H_k$ at specific sample points.  
\subsection{Preliminary results}
From Corollary \ref{coro:Hkapproximations}, we shall deduce
\begin{myprop}
\label{prop:Hksamplepoints} 
 Let $1/2 + i y_m = R e^{i \theta_m}$ where
  $\theta_m = \frac{\pi m}{k} \in [\frac{\pi}{3}, \frac{\pi}{2})$.  Then $\cos(k \theta_m) = (-1)^m$, and
 \begin{equation}
 (-1)^m H_k(1/2+iy_m) = O(k^{-1/15}) + 
 \left\{
 \begin{alignedat}{2}
   &2, \qquad &&y_m  \leq k^{2/5}, \\
   &2 + \gamma \Phi_0(r), \qquad &k^{2/5} < &y_m \leq 100 k^{1/2}, \\
    &r^{1/2} \exp(\frac{\pi}{4r})(1 + \gamma' \Phi_1(r)), \qquad &\frac{1}{100} k^{1/2} \leq  &y_m \leq k^{3/5}, \\
    &r^{1/2} (1 + O(k^{-1/15})), \qquad &k^{3/5} \leq  &y_m \leq k,
    \end{alignedat}
    \right.
\end{equation}
where $\gamma, \gamma'$ are real numbers of absolute value at most $1$.
\end{myprop}
From this one may conclude, in a quantitative way, that the sign of $H_k(1/2+iy_m)$ is the same as $(-1)^m$, for all $m$, provided $k$ is large enough for the error term to be negligible.  A key point is that the set of $r$'s for which $\Phi_0(r) < 2$ and $\Phi_1(r) < 1$ overlap, which is apparent from Figure \ref{fig:Phi0}.

\begin{proof}
For the values of $y_m$ with $y_m \leq 100 k^{1/2}$, this immediately follows from Corollary \ref{coro:Hkapproximations}, since $\cos(k \theta_m) = (-1)^m$,
so assume $y_m \gg k^{1/2}$.  Let $\phi_m = \frac{\pi}{2} - \theta_m = \frac{\pi}{2} - \frac{\pi m}{k}$, so that $\tan(\phi_m) = \frac{1}{2 y_m}$.  Since $y_m \gg k^{1/2}$, this means $\phi_m \ll k^{-1/2}$, and so $m = \frac{k}{2} - O(k^{1/2})$ (recall that $m$ runs over the integers in the interval $[k/3, k/2)$).  Say $m= \frac{k}{2}-p$, with $p=1,2, \dots, P$ with $P \ll \sqrt{k}$.  Then by standard Taylor approximations, we have
\begin{equation}
y_m = y_{\frac{k}{2}-p} = \frac{1}{2 \tan(\frac{\pi p}{k})} = \frac{1}{2 \frac{\pi p}{k} (1 + O(\frac{p^2}{k^2}))} = \frac{k}{2 \pi p} \Big(1 + O(k^{-2} p^2)\Big). 
\end{equation}
Since $k^{-2} p^2 \ll k^{-1}$, this value of $y$ satisfies \eqref{eq:yNapproxformula} with $N=p$, and so 
the conditions required to apply Corollary \ref{coro:Hkapproximations} with $y \gg k^{1/2}$ are met.  Thus, for $y_m \gg k^{1/2}$, 
$H_k(1/2 + i y_m)$ has the same sign as $(-1)^{p + \frac{k}{2}} = (-1)^m$, as desired.
\end{proof}

We shall need a result analogous to Lemma \ref{lemma:sinMonotonic}.
\begin{mylemma}
\label{lemma:cosMonotonic}
 The function $(2 \cos(\frac{\pi}{3} + \frac{\pi}{2X}))^X$ is monotonically increasing for $X \geq 3$.
\end{mylemma}
\begin{proof}
It suffices to show $f(X) := X \log(2 \cos(\frac{\pi}{3} + \frac{\pi}{2X}))$ is increasing for $X \geq 3$.  We have
\begin{equation}
 f'(X) = \tfrac{\pi}{2X} \cot(\tfrac{\pi}{3} + \tfrac{\pi}{2X}) + \log(2 \cos(\tfrac{\pi}{3} + \tfrac{\pi}{2X})),
\end{equation}
and we need to show $f'(X) \geq 0$.  With $Y = \frac{\pi}{2X}$, it suffices to show that $g(Y) := Y \cot(\frac{\pi}{3} + Y) + \log(2 \cos(\frac{\pi}{3} + Y) \geq 0$ for $Y \in [0, \pi/6]$.  We have $g(0) = 0$, and 
by direct calculation, $g'(Y) = Y \csc(\frac{\pi}{6} - Y)^2 \geq 0$, so we are done.
\end{proof}

\subsection{Proof of Proposition \ref{prop:BklLowerBoundAsymptotic}}
\label{section:samplepointsBkl}
We restate Proposition \ref{prop:BklLowerBoundAsymptotic} for convenience:
\begin{myprop}
\label{prop:BklLowerBoundAsymptoticSection5}
 Suppose $\ell$ is large.  Then $\Delta_{k, \ell}$ has at least $\lfloor \ell/6 \rfloor -1$ zeros on $x= 1/2$ if $\ell \equiv 2, 4 \pmod{6}$, and at least $\lfloor \ell/6] - 2$ zeros on $x=1/2$ if $\ell \equiv 0 \pmod{6}$.
\end{myprop}
\begin{proof}
By direct calculation, we have
 \begin{equation}
  \Delta_{k, \ell} = E_{k} E_{\ell} - E_{k+\ell} = (E_k - 1) + (E_\ell-1) + (E_k-1)(E_\ell-1) - (E_{k+\ell}-1).
 \end{equation}
Recalling the definition \eqref{eq:GandHdefinitions}, we have
\begin{equation}
\label{eq:zkpluslDeltakldef}
|z|^{k+\ell} \Delta_{k, \ell} = |z|^{k} H_\ell(z) + |z|^l H_k(z) + H_{k}(z) H_\ell(z) - H_{k+\ell}(z).
\end{equation}  


Suppose that $Q_{\ell}$ is the number of integers $m$ such that $\theta_m = \frac{\pi m}{\ell} \in (\frac{\pi}{3}, \frac{\pi}{2})$.
Let $\ell = 6q + a$, where $a \in \{0,2,4 \}$, so that $m \in (2q+\frac{a}{3}, 3q + \frac{a}{2})$.  Write $m = 2q + d$, so that $d$ ranges over the following integers:
\begin{equation}
\label{eq:rangeofrs}
\begin{cases}
1 \leq d \leq q-1, \qquad &\text{if } a=0, \\
1 \leq d \leq q, \qquad &\text{if } a =2, \\
2 \leq d \leq q+1, \qquad &\text{if } a=4.
\end{cases}
\end{equation}
Therefore,
\begin{equation}
 Q_{\ell} = \begin{cases}
 q-1, \qquad &\text{for } a=0, \\
 q, \qquad &\text{for } a =2, 4.
\end{cases}
\end{equation}
Our goal is to show that $\Delta_{k,\ell}$ exhibits sign changes along the points $1/2 + i y_m = R e^{i \theta_m}$, which then implies Proposition \ref{prop:BklLowerBoundAsymptoticSection5}
 
Suppose $k \geq \ell$.  Then we rearrange \eqref{eq:zkpluslDeltakldef} as
\begin{equation}
\label{eq:DeltaklapproximateformulaOnSide}
 |z|^{k+\ell} \Delta_{k, \ell}(z) = |z|^k \Big(H_\ell(z) + \frac{H_k(z)}{|z|^{k-\ell}}\Big) + H_k(z) H_\ell(z) - H_{k+\ell}(z).
\end{equation}
If $y \gg 1$, with a large enough implied constant (not depending on $k$ or $\ell$), then we claim that $|z|^{k+\ell} \Delta_{k, \ell}(1/2 + iy_m)$, with $1/2 + i y_m = r^{i \theta_m}$, $\theta_m = \frac{\pi m}{\ell} \in (\frac{\pi}{3}, \frac{\pi}{2})$, has the same sign as $(-1)^m$, provided that $\ell$ is large enough.  
For this, we use Corollary \ref{coro:Hktrivialbound} which then implies
\begin{equation}
 |z|^{k+\ell} \Delta_{k, \ell}(z) = |z|^k \Big(H_\ell(z) + O\Big(\frac{1 + \frac{y}{\sqrt{k}}}{y^2} \Big) \Big) + O((1 + k^{-1/2} y)(1 + \ell^{-1/2} y)),
\end{equation}
under the extra assumption $k > \ell$, while if $k=\ell$, we simply write
\begin{equation}
 |z|^{2 \ell} \Delta_{\ell,\ell}(z) = 2|z|^\ell H_\ell(z) + O((1 + \ell^{-1/2} y)^2).
\end{equation}
In either case, Proposition \ref{prop:Hksamplepoints} shows that $\Delta_{k,\ell}(1/2 + iy_m)$ has the same sign as $(-1)^m$, for $\ell$ large enough, and $y$ large enough.

Now suppose that $y \ll 1$.  Then we temporarily forget that $y= y_m$, and refer to \eqref{eq:EisensteinSideApproxsmallySecondForm}, 
which gives
\begin{equation}
\label{eq:DeltaklapproximateformulaOnSideWithTrigFunctions}
 |z|^{k+\ell} \Delta_{k, \ell}(z) = 2|z|^k P_{k,\ell}(\theta)
+ O(\exp(-\ell^{1/6})),
\end{equation}
where $z = 1/2 + iy = R e^{i \theta}$, and 
\begin{equation}
\label{eq:PkellDef}
P_{k,\ell}(\theta) = \cos(\ell \theta) + \frac{ \cos(k \theta)}{|z|^{k-\ell}}  + \frac{2 \cos(k \theta) \cos(\ell \theta) -  \cos((k+\ell)\theta)}{|z|^k}.
\end{equation}
Now we express $P_{k,\ell}(\theta_m)$ in simplified form by a change of variables.  
In terms of the variables $q,a,d$, we have
$\cos(\ell \theta_m) = (-1)^m = (-1)^d$,
\begin{equation}
\cos(k \theta_m) = \cos(\ell \theta_m + (k-\ell) \theta_m) = (-1)^d \cos((k-\ell)\theta_m), 
\end{equation}
and similarly $\cos((k+\ell)\theta_m) = \cos((k-\ell)\theta_m)$.  Let us also write $k-\ell = D$.
Combining these formulas, we derive
\begin{equation}
(-1)^d P_{k,\ell}(\theta_m) = 1 + \frac{\cos(D \theta_m)}{|z|^D} (1 + |z|^{-\ell} (-1)^d).
\end{equation}
We also have
\begin{equation}
\theta_m= \theta_{2q+d} = \frac{\pi (2q+d)}{6q+a}  = \frac{\pi}{3} + \frac{\pi}{\ell} (d - \frac{a}{3}),
\end{equation}
and
\begin{equation}
|z|^{-1} = 2 \cos(\theta_m).
\end{equation}
Therefore, with shorthand $x = \frac{\pi}{\ell}(d-\frac{a}{3})$, we have $\theta_m = \frac{\pi}{3} + x$, and 
\begin{equation}
\label{eq:Pkellthetamformula}
(-1)^d P_{k,\ell}(\theta_m) = 1 + (2 \cos(\tfrac{\pi}{3} + x))^D \cos(\tfrac{D \pi}{3} + Dx) [1 + (-1)^d (2 \cos(\tfrac{\pi}{3} + x))^{\ell}].
\end{equation}
We next claim the following.
\begin{myprop}
\label{prop:PkellLowerBound}
For all $m$ with $\theta_m \in (\tfrac{\pi}{3}, \tfrac{\pi}{2})$, we have
\begin{equation}
\label{eq:PkellLowerBound}
(-1)^m P_{k,\ell}(\theta_m) \geq 0.17,
\end{equation}
\end{myprop}
For $\ell$ sufficiently large, using Proposition \ref{prop:PkellLowerBound} in \eqref{eq:DeltaklapproximateformulaOnSideWithTrigFunctions} completes the proof of Proposition \ref{prop:BklLowerBoundAsymptoticSection5}.
\end{proof}

\begin{proof}[Proof of Proposition \ref{prop:PkellLowerBound}]
First suppose $D \equiv 0 \pmod{6}$.  If in addition, $x \leq \frac{\pi}{2D}$ then $\cos(\frac{\pi D}{3} + Dx) = \cos(Dx) \geq 0$, so $(-1)^m P_{k,\ell}(\theta_m) \geq 1$, which is even stronger than claimed; here we used that the expression in square brackets in \eqref{eq:Pkellthetamformula}  is $\geq 0$.  If instead $x > \frac{\pi}{2D}$, then
\begin{equation}
(2 \cos(\tfrac{\pi}{3} + x))^D \leq (2 \cos(\tfrac{\pi}{3} + \tfrac{\pi}{2D}))^D
< 0.07,
\end{equation}
where this final inequality arises as follows.  
By Lemma \ref{lemma:cosMonotonic}, it suffices to evaluate the limit as $D \rightarrow \infty$.  For this, we have
\begin{equation}
\lim_{D \rightarrow \infty} (2 \cos(\tfrac{\pi}{3} + \tfrac{\pi}{2D}))^D = 
\lim_{D \rightarrow \infty} \exp(D \log(1- \tfrac{\sqrt{3} \pi}{2 D} + O(D^{-2}))) = \exp(-\tfrac{\sqrt{3} \pi}{2}) = 0.0658\dots.
\end{equation}
Therefore, we derive
\begin{equation}
(-1)^d P_{k,\ell}(\theta_m) \geq 1 - (0.07)(2) = 0.86,
\end{equation}
consistent with \eqref{eq:PkellLowerBound}.

Next suppose $D \equiv 2, 4 \pmod{6}$.  In this case, if $x \leq \frac{\pi}{12D}$, then $\cos(\frac{\pi D}{3} + Dx) \geq - \frac{\sqrt{2}}{2}$.  We also have, using $x \geq \frac{\pi}{3\ell}$ which in turn follows from $d-\frac{a}{3} \geq \frac{1}{3}$, that
\begin{equation}
(2 \cos(\tfrac{\pi}{3} + x))^{-\ell} \leq (2 \cos(\tfrac{\pi}{3} + \tfrac{\pi}{3\ell}))^{-\ell}  = \Big[(2 \cos(\tfrac{\pi}{3} + \tfrac{\pi}{2(3\ell/2)}))^{-3\ell/2} \Big]^{2/3} \leq (0.0658\dots)^{2/3} = 0.16\dots. 
\end{equation}
Combining these bounds, we derive that
\begin{equation}
(-1)^d P_{k,\ell}(\theta_m) \geq 1 - \tfrac{\sqrt{2}}{2} (1 + 0.17) > 0.17.
\end{equation}
Finally, consider the case $x > \frac{\pi}{12D}$.  Then
\begin{equation}
(2 \cos(\tfrac{\pi}{3} + x))^D \leq [(2 \cos(\tfrac{\pi}{3} + \tfrac{\pi}{12D}))^{6D}]^{1/6} \leq \exp(-\tfrac{\sqrt{3} \pi}{12}) = 0.635\dots.
\end{equation}
Therefore, we have
\begin{equation}
(-1)^d P_{k,\ell}(\theta_n) \geq 1 - (0.64)(1+0.17) > 0.25. \qedhere
\end{equation}
\end{proof}

%
%

\subsection{The extra zero}
\label{section:extrazeroOnSide}
In this section, we prove Theorem \ref{thm:Bkvalue}.  Comparing with Proposition \ref{prop:BklLowerBoundAsymptotic}, we see that when $\ell \equiv 0,4 \pmod{6}$, we need to produce an additional sign change, under appropriate conditions on $k,\ell$.  Consulting Section \ref{section:samplepointsBkl}, we see that if $a=0$, the sample point nearest to $\frac{\pi}{3}$ is $\theta_{2q+1} = \frac{\pi}{3} + \frac{\pi}{\ell}$, and $P_{k,\ell}(\theta_{2q+1}) < 0$.  Meanwhile, for $a=4$, the nearest point is $\theta_{2q+2} = \frac{\pi}{3} + \frac{2\pi}{3\ell}$, and $P_{k,\ell}(\theta_{2q+2}) > 0$.

Let us first examine the easiest cases where $\frac{\pi}{3}$ is an acceptable sample point.  If $k + \ell \not \equiv 0 \pmod{6}$, then $\Delta_{k,\ell}(e^{\pi i/3}) = 0$, so we need to assume $k + \ell \equiv 0 \pmod{6}$.  Then by \eqref{eq:PkellDef}, we have
\begin{equation}
P_{k,\ell}(\tfrac{\pi}{3}) =  2 \cos^2(\tfrac{\pi \ell}{3}) + 2 \cos(\tfrac{\pi \ell}{3}) -1 = \begin{cases}
3, \qquad &\ell \equiv 0 \pmod{6} \\
-1.5, \qquad &\ell \equiv 4 \pmod{6}.
\end{cases}
\end{equation}
This furnishes the desired additional zero.

Now we examine the harder cases where $k + \ell \not \equiv 0 \pmod{6}$.  
First suppose $k+ \ell \equiv 4 \pmod{6}$, which is the case where $\Delta_{k,\ell}$ has a single zero at $e^{\pi i/3}$.
Recall that
\begin{equation}
|z|^{\ell} \Delta_{k,\ell} = H_{\ell}(z) + \frac{H_k(z)}{|z|^{k-\ell}} + \frac{1}{|z|^k} (H_k(z) H_{\ell}(z) - H_{k+\ell}(z)),
\end{equation}
that $H_k(1/2 + iy ) = 2 \cos(k\theta) + O((37/17)^{-k/2})$, so by Corollary \ref{coro:HderivativeEstimate}, that $\frac{d}{d\theta} H_k(1/2 + iy) = -2k \sin(k \theta) + O(2^{-k/2})$ (and likewise for the second derivative, though this is not necessary for this case).  Gathering these facts, we derive
\begin{equation}
\label{eq:DeltaklDerivativeEstimate}
\frac{d}{d \theta} |z|^{\ell} \Delta_{k,\ell}(1/2 + iy) \vert_{\theta = \pi/3} = \frac{d}{d \theta} P_{k,\ell}(\theta) \vert_{\theta = \pi/3} + O(k \thinspace 2^{-\ell/2}). 
\end{equation}

By a direct calculation (using Mathematica for simplicity), we derive
\begin{equation}
P_{k,\ell}'(\tfrac{\pi}{3}) = 
\begin{cases}
\sqrt{3} (2\ell -k), \qquad \ell \equiv 4 \mymod{6}, \thinspace k \equiv 0  \mymod{6} \\
\sqrt{3}(2k-\ell), \qquad  \ell \equiv 0 \mymod{6}, \thinspace k \equiv 4 \mymod{6}.
\end{cases}
\end{equation}
When $a=0$, that is, $\ell \equiv 0 \pmod{6}$, then $P'_{k,\ell}(\frac{\pi}{3}) > 0$, since $k \geq \ell$, while the nearest sample point to $\pi/3$ had a negative value of $P_{k,\ell}$.  Therefore, there exists a positive value of $P_{k,\ell}$ near to $\pi/3$.  This exhibits an additional sign change.

For $a=4$, that is, $\ell \equiv 4 \pmod{6}$, then $P_{k,\ell}'(\frac{\pi}{3}) < 0$ for $k > 2 \ell$ (which, by the way, implies $k \geq 2 \ell + 4$), while the nearest sample point has a positive value of $P_{k,\ell}$, so this will provide an additional sign change here also, provided $k 2^{-\ell/2}$ is small compared to $k-2\ell$.  However, this is always true, since we assume $\ell$ is sufficiently large.
In all, this provides the additional zero for sufficiently large $k$, $\ell$.

Similarly, for the second derivative, we derive
\begin{equation}
\label{eq:Pkdoubleprimeformula}
P_{k,\ell}''(\tfrac{\pi}{3}) = 
\begin{cases}
2(2k^2  + k (-2\ell + 1) - \ell^2  -  \ell), \qquad &\ell \equiv 0 \mymod{6}, \thinspace k \equiv 2  \mymod{6} \\
2(-k^2 + k(4\ell - 1) - \ell^2  -  \ell), \qquad  &\ell \equiv 4 \mymod{6}, \thinspace k \equiv 4 \mymod{6}.
\end{cases}
\end{equation}
Of course, the first derivative vanishes in this case.
Consider the case $\ell \equiv 0 \pmod{6}$.  Then the polynomial value satisfies
\begin{equation}
\label{eq:PkdoublePrimePolynomialExpression}
2k^2 + k(-2\ell + 1) -\ell^2 -  \ell = 2 ( k - x_{+})(k-x_{-}),
\end{equation}
where 
\begin{equation}
x_{\pm} = \frac{\ell-1 \pm \sqrt{3 \ell^2 - 1}}{2}.
\end{equation}
We see that if $k > x_{+}$, which is precisely the condition $k \geq \stabp_j(\ell)$ in this case, then certainly $P_{k,\ell}''(\frac{\pi}{3}) \geq 2$.  This is much larger than $k^2 2^{-\ell/2}$ unless $\ell \ll \log k$,  but in this case, we easily see from the left hand side of \eqref{eq:PkdoublePrimePolynomialExpression} that $P_{k,\ell}''(\pi/3) \gg k^2$, which 
is large compared to $k^2 2^{-\ell/2}$ for $\ell$ large.  

A similar calculation holds for $\ell \equiv 4 \pmod{6}$. In this case, the roots of the polynomial are
\begin{equation}
\frac{4 \ell -1 \pm \sqrt{12 \ell^2 -12 \ell + 1}}{2},
\end{equation}
which leads to the formula in \eqref{eq:stabpdef}.  We omit the details, since it is nearly identical to the previous case.

\subsection{The exceptional extra zero on the arc}
Using the same ideas as in Section \ref{section:extrazeroOnSide}, we may produce an additional zero on the arc, when it exists.
\begin{myprop}
 Suppose that $L \leq \ell \leq k < \stabp_j(\ell)$.  Then $A_{k,\ell} \geq N_{k,\ell}' + 1$.
\end{myprop}
There are only three cases where $\stabp_j(\ell) \neq \ell$, given by the conditions on the right hand side of \eqref{eq:stabpdef}.  
\begin{proof}
The proof here is rather similar to that appearing in Section \ref{section:extrazeroOnSide}, in that we will estimate the first and second derivatives of $M_{k,\ell}(\theta)$ at $\pi/3$.  We begin with the claim that
\begin{equation}
\frac{d^a}{d \theta^a} (F_k(\theta) F_{\ell}(\theta) - F_{k+\ell}(\theta)) \vert_{\theta=\pi/3} = \frac{d^a}{d \theta^a} M_{k,\ell}(\theta) \vert_{\theta=\pi/3} + O(k^a 2^{-\ell/2}),
\end{equation}
for $a=1,2$.  This is similar to, yet even easier than, 
\eqref{eq:DeltaklDerivativeEstimate},
since the use of the chain rule is simpler since $G_k(z)$ is holomorphic (also one should recall $F_k(\theta) = G_k(e^{i\theta})$). 

First, suppose that $k \equiv 0 \pmod{6}$ and $\ell \equiv 4 \pmod{6}$, whence $j \in \{2,8\}$.  In this case, there is a single zero of $\Delta_{k,\ell}$ at $e^{i\pi/3}$, and correspondingly we have $M_{k,\ell}(\frac{\pi}{3}) = 0$.  By a tedious but direct calculation, we have
\begin{equation}
 M_{k,\ell}'(\tfrac{\pi}{3}) = \sqrt{3}(-k + 2\ell), 
\end{equation}
provided $j=2$.  If $j=8$, the sign here is reversed.  For $j=2$, the nearest sample point to $\pi/3$ had $m=2n+1$, whence $M_{k,\ell}(\theta_{2n+1}) < 0$, while for $j=8$, the sign here is switched.  Thus, if $k < 2 \ell$, that is, $k < \stabp_2(\ell)$, then $M_{k,\ell}'(\frac{\pi}{3}) > 0$.  By similar reasoning to that occuring in Section \ref{section:extrazeroOnSide}, one can show that $\frac{d}{d\theta} (F_k(\theta) F_{\ell}(\theta) - F_{k+\ell}(\theta)) \vert_{\theta=\pi/3} > 0$, at least, provided $\ell$ is sufficiently large.  This provides the desired additional zero.  The same argument works for $j=8$ with signs reversed.

Next assume that $\ell \equiv 0 \pmod{6}$ and $k \equiv 2 \pmod{6}$, or $k \equiv \ell \equiv 4 \pmod{6}$.  In these cases, there is a double zero at $\pi/3$, and we have
\begin{multline}
\label{eq:Mkdoubleprime}
 M_{k,\ell}''(\tfrac{\pi}{3}) = -\tfrac12 (k-\ell)^2 \cos(\tfrac{(k-\ell)\pi}{6})
 + i^{-\ell} \cos(\tfrac{k\pi}{6}) [\tfrac{-k^2+3\ell^2 + 4\ell}{2}]
 \\ 
 i^{-k} \cos(\tfrac{\ell \pi}{6}) [\tfrac{-\ell^2+3k^2 + 4k}{2}]
 + \sqrt{3} k \ell [i^{-\ell} \sin(\tfrac{k\pi}{6}) + i^{-k} \sin(\tfrac{\ell \pi}{6})].
\end{multline}
If $\ell \equiv 0 \pmod{6}$, $k\equiv 2 \pmod{6}$, and say $j=2$, then by a direct calculation, we have
\begin{equation}
 M_{k,\ell}''(\tfrac{\pi}{3}) =
 -2k^2 + 2k(\ell+1) + \ell^2 + \ell.
\end{equation}
If $j=8$, the sign is switched.  As in the previous case, the nearest sample point has a negative value of $M_{k,\ell}$, so if $M_{k,\ell}'' > 0$ this will produce an extra zero (provided the error term is negligible of course).  
Up to a factor $-2$, this formula agrees with \eqref{eq:Pkdoubleprimeformula}, so this provides the desired additional zero for $\ell$ large.

Simplifying \eqref{eq:Mkdoubleprime} in the case $k\equiv \ell \equiv 4 \pmod{6}$, and with $j=0$, we obtain
\begin{equation}
 M_{k,\ell}''(\tfrac{\pi}{3}) = -k^2 + k(4\ell - 1) - \ell^2 - \ell .
\end{equation}
Again, when $j=6$, the sign here is switched, and again we obtain the same polynomial as in \eqref{eq:Pkdoubleprimeformula}, up to a factor $-2$.  This provides the desired additional zero. 
\end{proof}

\end{document}